\documentclass[a4paper,10pt]{amsart}
\usepackage[utf8]{inputenc}
\usepackage{amsthm}
\usepackage{amsmath}
\usepackage{comment}
\usepackage{bm}
\usepackage{amsfonts}
\usepackage{amssymb}
\usepackage{mathtools}
\usepackage{stmaryrd}
\usepackage{mathrsfs}
\usepackage{setspace}
\usepackage{xcolor}
\usepackage{textgreek}
\usepackage{todonotes}
\usepackage{caption}
\usepackage{tikz}
\usetikzlibrary{arrows.meta}
\usepackage{pgfplots}
    \pgfplotsset{
        compat=1.15,
        width=8cm,
    }
\usepackage[a4paper, left=2.8cm, right=2.8cm, top=3cm, bottom=3cm]{geometry}

\usepackage{thmtools} 
\usepackage{bm}
\usepackage{keyval}

\usepackage[pdfdisplaydoctitle,colorlinks,breaklinks,urlcolor=blue,linkcolor=blue,citecolor=blue]{hyperref} 
\mathtoolsset{showonlyrefs}

\newcommand{\R}{\mathbb{R}}

\newcommand{\Z}{\mathbb{Z}}

\newcommand{\NN}{\mathbb{N}}

\newcommand{\TT}{\mathbb{T}}

\newcommand{\vertiii}[1]{{\left\vert\kern-0.25ex\left\vert\kern-0.25ex\left\vert #1 
    \right\vert\kern-0.25ex\right\vert\kern-0.25ex\right\vert}}

\numberwithin{equation}{section}
\newtheorem{theorem}{Theorem}[section]

\newtheorem{definition}[theorem]{Definition}
\newtheorem{corollary}[theorem]{Corollary}

\newtheorem{lemma}[theorem]{Lemma}

\newtheorem{proposition}[theorem]{Proposition}

\theoremstyle{definition}

\title[Weak Sard and Anomalous Dissipation]
{Failure of the Weak Sard property \\ without Anomalous Dissipation}

\author[U. Pappalettera]{Umberto Pappalettera}
\address{Departement Mathematik und Informatik, Universit\"at Basel, Spiegelgasse 1, CH-4051 Basel, Switzerland.} 
\email{umberto.pappalettera@unibas.ch}

\date\today

\begin{document}

\begin{abstract}
For every $\alpha\in(0,1)$ we construct an autonomous, divergence-free vector field
$u \in C^\alpha_c(\mathbb{R}^2,\mathbb{R}^2)$ that does not have the weak Sard property and, nonetheless, does not induce anomalous dissipation of $L^2(\mathbb{R}^2)$ norm for solutions to the associated advection-diffusion equation. This disproves a conjecture proposed by Bagnara, Boutros, De Lellis and Mayboroda in \cite{BaBoDeMa26}.
\end{abstract}

\maketitle

\section{Introduction}
\label{sec:introduction}

Let $\alpha \in (0,1)$ and $u \in C^\alpha(\R^2,\R^2)$ be a H\"older continuous, bounded, autonomous vector field with zero divergence in the sense of distributions.
In this paper we are concerned with the advection-diffusion equation
\begin{align} \label{eq:advection-diffusion}
 \partial_t\theta^\kappa+u\cdot\nabla\theta^\kappa
  =\kappa\Delta\theta^\kappa,
  \quad
  \kappa \in (0,1),
\end{align}
with initial condition $\theta_0^\kappa \equiv \theta_0 \in L^2(\R^2)$, and more specifically we are interested in the problem of \emph{anomalous dissipation} for this class of velocity fields. 

By \cite{BoCiCr24}, \eqref{eq:advection-diffusion} has a unique \emph{parabolic} solution $\theta^\kappa \in C([0,1],L^2(\R^2)) \cap L^2([0,1],H^1(\R^2))$. 
We say that \eqref{eq:advection-diffusion} admits anomalous dissipation if there exists an initial condition $\theta_0 \in L^2(\R^2)$ such that
\begin{align} \label{eq:definition.anomalous.diss}
    \limsup_{\kappa \downarrow 0} \kappa \int_0^1 \| \nabla \theta^\kappa_t \|_{L^2}^2 dt >0.
\end{align}

Heuristically, \eqref{eq:advection-diffusion} can exhibit anomalous dissipation only if the transport term $u \cdot \nabla$ transfers enough energy of $\theta^\kappa$ to small spatial scales, thus enhancing dissipation due to the diffusion term $\kappa \Delta$, in a way that quantitatively balances the infinitesimal factor $\kappa \downarrow 0$ in \eqref{eq:definition.anomalous.diss}. 
As such, it is considered an interesting mechanism mimicking the cascade of energy to small spatial scales that we observe in turbulent fluids.

In fact, the problem of anomalous dissipation can be considered also for time-dependent velocity fields $u$ and most mathematically rigorous results address exactly this case: we mention, among others and without claiming of exhaustivity, \cite{DrElIyJe22,CoCrSo23,ElLi24,ArVi25,BuSzWu23,HoPaZhZh25,HeRo25,BuSzWu26}.

On the other hand, in this work we focus on the case of autonomous velocity fields. In this setting, the problem of constructing a velocity field inducing anomalous dissipation becomes much more difficult due to additional structural constraints imposed by the Hamiltonian nature of the velocity field. 
More precisely, for every divergence-free $u \in C^{\alpha}(\R^2,\R^2)$ there exists a Hamiltonian function $H \in C^{1,\alpha}(\R^2)$ such that   
\begin{align} \label{eq:hamiltonian}
    u = \nabla^\perp H := (-\partial_2 H,\partial_1 H )
\end{align}
and moreover $H$ is unique up to additive constants. 

As already observed in \cite{JoSo24+}, anomalous dissipation can only happen if the (backwards) inviscid problem \eqref{eq:advection-diffusion} with $\kappa=0$ admits non-unique bounded solutions \cite{Ro24}, which by \cite{AlBiCr14} is equivalent to failure of the \emph{weak Sard property} for the Hamiltonian of $u$.

The classical Sard property, valid e.g. for every $C^2$ function $H:\R^2 \to \R$, states that the image via $H$ of the \emph{critical points}
\begin{align}
 S := \{ x \in \R^2 \,:\, \nabla H = 0 \}   
\end{align}
has zero Lebesgue measure $\mathscr{L}^1(H(S))=0$.
The weak Sard property is a measure-theoretic extension of this notion.
Let $E^*$ be the union, over all $h\in\R$, of the connected components of
$H^{-1}(\{h\})$ having positive Hausdorff $\mathscr{H}^1$ measure. Following
\cite{AlBiCr13,AlBiCr14}, we say that $H \in C^{1,\alpha}(\R^2)$ has the weak Sard property if
\begin{align} \label{eq:definition.weaksard}
    H_{\sharp}(\mathbf{1}_{S \cap E^*} \mathscr{L}^2 ) \perp  \mathscr{L}^1.
\end{align}
Notice that this property is independent of the choice of $H$ satisfying \eqref{eq:hamiltonian}. In particular, we say that $u$ satisfies the weak Sard property if one (or equivalently any) of its Hamiltonians satisfies the weak Sard property above.

We also mention the work \cite{DrGoGu26}, where the weak Sard property is compared
with the stronger condition
\begin{align}
  H_\sharp \left(\mathbf{1}_{S}\mathscr{L}^2\right)
 \perp \mathscr{L}^1  ,
\end{align}
called \emph{relaxed} Sard property. The two notions are shown to be
distinct for Hamiltonians $H \in C^{1,\alpha}(\R^2)$ for every $\alpha<1$. 
A further recent example appears in \cite{DrGu26}, where the authors construct a vector field which satisfies the chain-rule property but not the
renormalization property. In view of \cite{AlBiCr14}, this example necessarily fails the weak Sard property. 

By \cite[Remark 1.3]{JoSo24+}, failing the weak Sard property is, in fact, a necessary condition for anomalous dissipation. 
Despite this strong structural constraint, the authors of \cite{JoSo24+} succeeded in producing an example of velocity field $u$ as above admitting anomalous dissipation.
Their construction, somehow reminiscent of the construction from \cite[Section 5]{AlBiCr13} to prove strong non-locality of the divergence operator, allows to show anomalous dissipation for \eqref{eq:advection-diffusion} as a consequence of the Fluctuation-Dissipation relation \cite{DrEy17} and spontaneous stochasticity for the associated Lagrangian trajectories.

It is tempting to try repeating the approach of \cite{JoSo24+} for other velocity fields failing the weak Sard property, or equivalently admitting non-unique solutions for the transport equation. 
Heuristically, one would like to find a sequence of diffusivities $\{\kappa_q\}_{q \in \NN} \subset (0,1)$ with $\kappa_q \downarrow 0$ such that the associated Lagrangian trajectories have positive statistical variance uniformly in $q \in \NN$. The idea is that the additive Brownian noise in the Lagrangian trajectories ``explores'' enough of the physical space around critical points that at least some trajectories will always separate.
The authors in \cite{BaBoDeMa26} pushed this analogy further by advancing the conjecture that failure of the relaxed Sard property is, in fact, equivalent to anomalous dissipation.
The purpose of this work is to give an answer to this conjecture in the negative.
Specifically, we prove the following:
\begin{theorem} \label{thm:main}
For every $\alpha \in (0,1)$ there exists a compactly supported, H\"older continuous, autonomous, divergence-free vector field $u \in C^\alpha_c(\R^2,\R^2)$ without the weak Sard property \eqref{eq:definition.weaksard}, and such that for every $\theta_0 \in L^2(\R^2)$ the unique parabolic solution of \eqref{eq:advection-diffusion} satisfies
\begin{align} \label{eq:no.anomalous.dissipation}
    \lim_{\kappa \downarrow 0} \kappa \int_0^1 \| \nabla \theta^\kappa_t \|_{L^2}^2 dt = 0.
\end{align}
\end{theorem}

Since the relaxed Sard property is in general stronger than the weak Sard property, our result disproves \cite[Conjecture 1.3]{BaBoDeMa26} up to rescaling and extending $u \in C^\alpha_c(\R^2,\R^2)$ by periodicity to a function $u \in C^\alpha(\TT^2,\R^2)$. 

\subsection{Idea of the proof}
The proof of \autoref{thm:main} is based on an explicit construction of the
Hamiltonian $H$, rather than directly of the velocity field $u$. 
The construction is carried out by parametrizing an annulus-shaped region $\mathcal A \subset \R^2$ with a smooth, measure-preserving diffeomorphism
\begin{align}
\Psi_0:\TT \times I \to \mathcal{A}    
\end{align}
which is affine on some rectangle $\mathcal{R}_0$.  
We then compose $\Psi_0$ with a sequence of smooth measure-preserving diffeomorphisms $\{\beta_q\}_{q \in \NN}$ of $\TT \times I$, each of which contracts width and expands height by a large factor $V \gg 1$ on a large subregion of $\mathcal{R}_0$, and acts as the identity on the complement of $\mathcal{R}_0$.
The resulting smooth diffeomorphisms $\Psi_q$ induce a Hamiltonian and a velocity field via the formulae (here $\pi_2:\TT \times I \to I$ denotes the projection on the second component)
\begin{align}
    H_q := \pi_2 \circ \Psi_q^{-1},
    \quad
    u_q := (\partial_1 \Psi_q) \circ\Psi_q^{-1}.
\end{align}

We then extend each $H_q$ to a function on the whole plane $\R^2$ and denote their limit as $q \to \infty$ by $H$.

To show that the weak Sard property for $H$ fails, we exhibit a set $\mathcal{P}_\infty \subset \TT \times I$ of positive measure on which each $\beta_q$ behaves as described above. 
Since width is increasingly contracted at each step of the iteration, the quantity $\partial_1 \Psi_q$ becomes smaller and smaller at every point $z \in \mathcal{P}_\infty$, and therefore by the previous formula $u_q$ tends to vanish on the set $\Psi_q(\mathcal{P}_\infty)$.
Thus for $\Psi := \lim_{q \to \infty} \Psi_q$ it holds $\Psi(\mathcal{P}_\infty)\subset S$.
Furthermore, in \autoref{prop:sard} we show that $\Psi(\mathcal{P}_\infty)\subset E^*$ and  
\begin{align} 
  H_\sharp (\mathbf{1}_{\Psi(\mathcal{P}_\infty)} \mathscr{L}^2 )
  =
  m \mathbf{1}_{U}(h) dh
\end{align}
for some positive number $m>0$ and a measurable set $U$ satisfying $\mathscr{L}^1(U)>0$, contradicting \eqref{eq:definition.weaksard}. 

On the other hand, the absence of anomalous dissipation follows from a quantitative
approximation argument. The argument is based on a criterion (\autoref{prop:criterion.no.diss}) that compares the solution of \eqref{eq:advection-diffusion} with the solution of inviscid transport with velocity $u_q$ via the triangle inequality
\begin{align}
    \| S^{\kappa}_H(t,\theta_0) - S^{0}_{H_q}(t,\theta_0) \|_{L^2}
    \leq
    \| S^{\kappa}_H(t,\theta_0) - S^{\kappa}_{H_q}(t,\theta_0) \|_{L^2}
    +
    \| S^{\kappa}_{H_q}(t,\theta_0) - S^{0}_{H_q}(t,\theta_0) \|_{L^2},
\end{align}
where we have denoted $S^\kappa_K : \R_+ \times L^2_x \to L^2_x$ the solution operator associated to \eqref{eq:advection-diffusion} with velocity $v = \nabla^\perp K$.
Each term on the right-hand side above can be separately controlled in terms of the quantity $\| H - H_q \|_{L^\infty}$ and of the space derivative of the inviscid flow $\Phi_q$ associated with $u_q$, see \autoref{lem:hamiltonian-comparison} and \autoref{lem:viscous-inviscid}.
For this we make a fundamental use of the fact that our construction makes the flow explicit in the coordinates determined by $\Psi_q$, namely:
\begin{align}
 \Phi_q(t,\Psi_q(s,h))=\Psi_q(s+t,h).   
\end{align}
Incidentally, this formula for the inviscid flow at step $q$ and the bound on $S^{\kappa}_H(t,\theta_0) - S^{0}_{H_q}(t,\theta_0)$ above characterize the unique zero-diffusivity limit of the advection-diffusion equation \eqref{eq:advection-diffusion}. Indeed, since $\Psi_q \to \Psi$ and $\Psi_q^{-1} \to \Psi^{-1}$ as $q\to \infty$, then $\theta^\kappa$ converges pointwise in time, as $\kappa \downarrow 0$, to the particular solution of the transport equation obtained by composing the initial condition with the inverse of the limiting flow $\Phi(t,\Psi(s,h)):=\Psi(s+t,h)$. 
This also clarifies, from the Lagrangian point of view, why anomalous dissipation is absent: the limiting flow prescribes how Lagrangian trajectories move when encountering the critical set $S$, preventing arbitrary pauses in $S$ and hence spontaneous stochasticity.

\subsection{Organization of the paper}
The paper is organized as follows. 
In \autoref{sec:parametrization} we construct the measure-preserving parametrization $\Psi_0$ and the maps $\{\beta_q\}_{q \in \NN}$.
In \autoref{sec:Hamiltonian} we define the limiting Hamiltonian and establish bounds on $\| H - H_q \|_{L^\infty}$ and $D\Phi_q$ needed to apply the criterion \autoref{prop:criterion.no.diss} preventing anomalous dissipation. 
In \autoref{sec:dissipation} we prove \autoref{prop:criterion.no.diss} and apply it to our construction, thus showing \eqref{eq:no.anomalous.dissipation}.
Finally, in \autoref{sec:weaksard} we show the failure of the weak Sard property.

\section*{Acknowledgements}
This project has received funding from the Swiss National Science Foundation under the SNSF Ambizione grant No. 233216.

\section{$\TT \times I$ parametrization of the annulus}
\label{sec:parametrization}

In this section we construct the parametrizations $\Psi_q :\TT \times I \to \mathcal{A}$ needed to define the Hamiltonian.
Here $\TT := \R/\Z$ denotes the circle and $I \subset \R$ is an open interval.
The underlying first parametrization $\Psi_0 :\TT \times I \to \mathcal{A}$ is constructed in \autoref{ssec:Psi_0} and the subsequent maps $\Psi_q$ are obtained by composition
\begin{align}
    \Psi_{q+1} := \Psi_0 \circ \beta_0 \circ \dots \circ \beta_{q}
\end{align}
where the maps $\beta_j$ are defined in \autoref{ssec:composition} and \autoref{ssec:recursion}.

\subsection{Smooth $\TT \times I$ parametrization}
\label{ssec:Psi_0}
Let $\gamma:\TT \to \R^2$ be a smooth simple closed curve, with constant speed $|\gamma'| \equiv 1$ and containing a straight segment, namely $\gamma'(s)$ is constant for every $s$ in a neighborhood of some $s_* \in \TT$.

In this subsection we want to define a tubular neighborhood $\mathcal{A} \subset \R^2$ of the image of $\gamma$, parametrized by points in a periodic strip $\TT \times I $, such that the parametrization (denoted below by $\Psi_0$) behaves as an affine transformation on a cartesian rectangle $\mathcal{R}_0 :=R_s \times R_h \subset \TT \times I $, see \autoref{fig:A}.

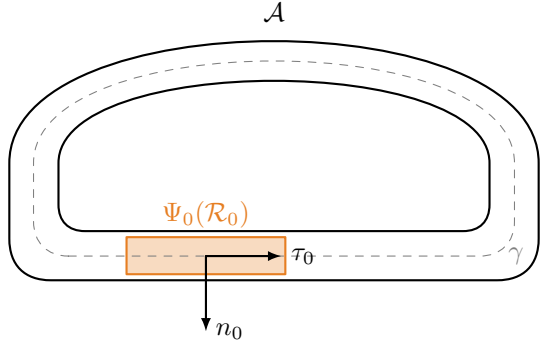
\begin{figure}[h]
\centering
\begin{tikzpicture}[
  x=1cm,
  y=1cm,
  line cap=round,
  line join=round
]
  \definecolor{coreorange}{RGB}{230,126,34}

  \filldraw[
    fill=coreorange!28,
    draw=coreorange,
    thick
  ]
    (1.55,0.08) rectangle (3.65,0.57);

  \draw[thick]
    (0.55,0) -- (6.45,0)
    .. controls (6.80,0) and (7,0.20) .. (7,0.55)
    -- (7,1.55)
    .. controls (7,3.70) and (0,3.70) .. (0,1.55)
    -- (0,0.55)
    .. controls (0,0.20) and (0.20,0) .. (0.55,0)
    -- cycle;

  \draw[thick]
    (1.00,0.65) -- (6.00,0.65)
    .. controls (6.22,0.65) and (6.35,0.78) .. (6.35,1.00)
    -- (6.35,1.55)
    .. controls (6.35,3.00) and (0.65,3.00) .. (0.65,1.55)
    -- (0.65,1.00)
    .. controls (0.65,0.78) and (0.78,0.65) .. (1.00,0.65)
    -- cycle;

  \draw[dashed,gray]
    (0.78,0.32) -- (6.22,0.32)
    .. controls (6.50,0.32) and (6.68,0.50) .. (6.68,0.78)
    -- (6.68,1.55)
    .. controls (6.68,3.35) and (0.32,3.35) .. (0.32,1.55)
    -- (0.32,0.78)
    .. controls (0.32,0.50) and (0.50,0.32) .. (0.78,0.32)
    -- cycle;

  \node[above,coreorange] at (2.60,0.57)
    {$\Psi_0(\mathcal{R}_0)$};

  \draw[-{Latex[length=1.8mm]},thick]
    (2.60,0.32) -- (3.60,0.32)
    node[right] {$\tau_0$};

  \draw[-{Latex[length=1.8mm]},thick]
    (2.60,0.32) -- (2.60,-0.68)
    node[right] {$n_0$};

  \node at (3.50,3.55) {$\mathcal{A}$};
  \node[gray] at (6.70,0.32) {$\gamma$};
\end{tikzpicture}

\caption{On the highlighted constant-speed straight segment,
$\kappa(s)\equiv 0$ and $h(s,r)=r$. Since $\Psi_0$ is affine
on $\mathcal{R}_0$, the image $\Psi_0(\mathcal{R}_0)$ is a rectangle
compactly contained in $\mathcal{A}$.}
\label{fig:A}
\end{figure}

For $s \in \TT$ let $\tau(s)$ denote the tangent versor to the curve $\gamma$ at the point $\gamma(s)$ and choose a normal versor $n(s)$ such that the basis of $\R^2$ $(n(s),\tau(s))$ is positively oriented.
Define the curvature $\kappa(s)$ by
\begin{align}
n'(s)=:-\kappa(s)\tau(s),
\end{align}
and the function
\begin{align}
  h(s,r):=r-\frac{\kappa(s)}2r^2,
  \quad
  r \in \R.
\end{align}
Take $0<\epsilon \ll1 $ sufficiently small that
\begin{align} \label{eq:partialh}
    \partial_r h (s,r) = 1-\kappa(s)r>0,
  \quad
  \forall s \in \TT, \quad
\forall r \in (-\epsilon,\epsilon).
\end{align}
Since $h(s,0)=0$ we can find open intervals $I \subset  \tilde{I}$ such that
\begin{align} \label{eq:assumptions.intervals}
    0 \in I,
    \quad
    \bar{I} \subset \tilde{I},
    \quad
    \tilde{I} \subset \{ h \in \R \,:\, \forall s \in \TT ,\, \exists r \in (-\epsilon,\epsilon) \mbox{ such that } h=h(s,r)\}.
\end{align}
The implicit function Theorem guarantees the existence of a smooth inverse map
\begin{align}
    r : \TT \times \tilde{I} \to (-\epsilon,\epsilon).
\end{align}
 
Define the map $\tilde\Psi_0 :\TT \times \tilde{I} \to \R^2$ by
\begin{align}
 \tilde{\Psi}_0(s,h)
  &:=
  \gamma(s) + r(s,h) n(s)
\end{align}
and $\Psi_0:=\tilde\Psi_0|_{\TT \times I}$.

\begin{lemma}
\label{lem:Psi_0.diffeo}
There exist open intervals $I \subset  \tilde{I}$ satisfying \eqref{eq:assumptions.intervals} such that the maps
\begin{align}
 \tilde\Psi_0 &:\TT \times \tilde{I}\to\tilde{\mathcal{A}},
 \\
\Psi_0 &:\TT \times I\to\mathcal A  
\end{align}
are smooth measure-preserving diffeomorphisms with their images
\begin{align}
 \tilde{\mathcal{A}}:=\tilde\Psi_0(\TT \times \tilde{I}),
  \qquad
  \mathcal A:=\Psi_0(\TT \times I).   
\end{align}
and $\bar{\mathcal A}\subset\tilde{\mathcal{A}}$.  
Moreover, $\Psi_0$ acts as an affine map on a rectangle $\mathcal{R}_0 = R_s \times R_h$ of positive width and height, that is compactly contained in $ \TT \times I$.
\end{lemma}

\begin{proof}
Let $I \subset  \tilde{I}$ be defined as above. 
By construction, $\tilde{\Psi}_0,\Psi_0$ are smooth. Chain rule gives
\begin{align}
\partial_h \tilde \Psi_0(s,h)
=
n(s) \partial_h r(s,h)  ,
\quad
\partial_s \tilde \Psi_0(s,h)
=
\tau(s) (1-\kappa(s) r(s,h))+ n(s) \partial_sr(s,h).
\end{align}
Using that $\det (n(s),\tau(s))=1$ for every $s \in \TT$ and \eqref{eq:partialh} we get 
\begin{align} \label{eq:det=1}
\det(\partial_h \tilde \Psi_0(s,h),\partial_s \tilde \Psi_0(s,h)) 
&=
\partial_h r(s,h)  (1-\kappa(s) r(s,h))
\\
&=
\partial_h r(s,h) \partial_r h (s,r(s,h)) = 1,
\end{align}
and therefore $\tilde{\Psi}_0$ is measure-preserving, as well as $\Psi_0$.
Let us show that they are diffeomorphisms onto their respective images.

Let us introduce a tubular neighborhood of $\gamma$ by setting for $0<\epsilon\ll 1$
\begin{align} 
F_\epsilon(s,r):=\gamma(s)+rn(s),
\quad
s \in \TT, \quad
r \in (-\epsilon,\epsilon).
\end{align}

For a sufficiently small $\epsilon$, the tubular neighborhood Theorem makes
$F_\epsilon$ a diffeomorphism onto its image. Then 
\begin{align}
    \tilde{\Psi}_0(s,h) = F_\epsilon(s,r(s,h) ) :  \TT \times \tilde{I}\to\tilde{\mathcal{A}}
\end{align}
is a diffeomorphism, too. 
A similar argument shows that its restriction $\Psi_0 :\TT \times I\to\mathcal A$ is a diffeomorphism. 
In addition, $\bar{\mathcal A}\subset\tilde{\mathcal{A}}$ because $\bar I\subset \tilde{I}$ by definition.

Finally, we have to show the existence of a rectangle $R_s \times R_h$ compactly contained in $\TT \times I $.

Recall that we have assumed that $\gamma$ contains a straight segment; here there are constant versors $\tau_0$, $n_0$, and a parameter value $s_* \in \TT$ such that for every $s$ in a neighborhood of $s_*$
\begin{align}
 \gamma(s)=\gamma(s_*)+(s-s_*)\tau_0,
  \quad
  \kappa(s) \equiv 0 ,  
  \quad 
  n(s) \equiv n_0.
\end{align}
Therefore $h(s,r)=r$ in a neighborhood of $ \{s_*\}\times (-\epsilon,\epsilon)$, and therefore for every $(s,h)$ therein
\begin{align} \label{eq:Psi.affine}
    \tilde{\Psi}_0(s,h)
  =
  \gamma(s_*)+(s-s_*)\tau_0+h n_0,
\end{align}
which is an affine function in $(s,h)$.  
Inverting this neighborhood of
$ \{s_*\}\times (-\epsilon,\epsilon)$ via the map $(s,h(s,r))$ and possibly intersecting with $\TT \times I$ gives the desired rectangle $\mathcal{R}_0 := R_s \times R_h$. 
\end{proof}

Up to applying a rotation, we can suppose without loss of generality that
\begin{align} \label{eq:extra.assumption.curve}
    \tau_0 = \begin{pmatrix}
        1 \\ 0
    \end{pmatrix},
    \quad
    n_0 = \begin{pmatrix}
        0 \\ -1
    \end{pmatrix}
\end{align}
so that the sides of the rectangle $\Psi_0(\mathcal{R}_0) \subset \R^2$ are parallel to the coordinate axes, as visualized in \autoref{fig:A}. 

\subsection{Rectangular quarter turns}
\label{ssec:quarter-turn}
Denote 
\begin{align}
 J:=
  \begin{pmatrix} 
  0&-1
  \\
  1&0
  \end{pmatrix}   
\end{align}
the matrix representing $90^\circ$ anticlockwise rotations of vectors in $\R^2$.
In this subsection we construct smooth diffeomorphisms of rectangles that act as the identity close to the boundary, while acting on the bulk of the rectangle as anticlockwise quarter turns that are conjugate to the rotation matrix $J$ by an affine transformation. The idea is similar to that of \cite[Section 3]{ElZl19} and \cite[Section 3]{Pa25}, with some technical differences in order to have fixed ratios between the bulk and the neighborhood of the boundary.

More specifically, let us denote the square
\begin{align}
\mathcal{Q}:=(-1,1)^2  
\end{align}
and the rectangle $\mathcal{R} \subset 2\mathcal{Q}$
\begin{align}
    \mathcal{R} := z_0 + A \mathcal{Q},
    \quad
   z_0 \in \mathcal{Q},
    \quad
    A:= \begin{pmatrix} 
  L_1  & 0
  \\
  0 & L_2
  \end{pmatrix},
  \quad
  L_1,L_2 \in (0,1).
\end{align}

We have the following:
\begin{lemma}
\label{lem:quarter.turn}
There exist constants $\delta_\star>0$ and $0<c_\star <C_\star$ such that, for every
$\delta \in (0,\delta_\star)$, there exists a smooth diffeomorphism $T_\delta : \mathcal{Q} \to \mathcal{Q}$ with the following properties:
\begin{itemize}
\item $T_\delta(z)=z$ for every $z \in \mathcal{Q}$ with $\mathrm{dist}(z,\partial \mathcal{Q})<c_\star \delta$; 
\item $T_\delta(z)=J z$ for every $z \in \mathcal{Q}$ with $\mathrm{dist}(z,\partial \mathcal{Q})>C_\star \delta$;
\item For every fixed integer $k \in \{1,2\}$, there exists
      $m_k<\infty$ such that
      \begin{align} \label{eq:bound.Ck.Tdelta}
        \|T_\delta\|_{C^k(\mathcal{Q})}
        +
       \|T_\delta^{-1}\|_{C^k(\mathcal{Q})}
        \lesssim \delta^{-m_k};
      \end{align}
\item 
$T_\delta$ preserves the Lebesgue measure on $\mathcal{Q}$.
\end{itemize}

Moreover, for every $\delta \in (0,\delta_\star)$ and rectangle $\mathcal{R}$ as above, the affine conjugate
\begin{align}
T_{\delta,\mathcal R}(z_0+Az)
  :=z_0+AT_\delta(z)    
\end{align}
preserves the Lebesgue measure on $\mathcal{R}$, and satisfies, for every component $i \in \{1,2\}$ and every multi index $a=(a_1,a_2)$ with $|a| = k \in \{1,2\}$
\begin{align}
        \|\partial^a (T_{\delta,\mathcal R})^i\|_{L^\infty(\mathcal{R})}
        +
       \|\partial^a (T_{\delta,\mathcal R}^{-1})^i\|_{L^\infty(\mathcal{R})}
        \lesssim \frac{L_i}{L_1^{a_1} L_2^{a_2}}\delta^{-m_k},
\end{align}
with implicit constant depending only on the implicit constant in \eqref{eq:bound.Ck.Tdelta}. 
In addition, $T_{\delta,\mathcal R}$ coincides with the identity in a neighborhood of $\partial \mathcal{R}$, while in a region $\mathcal{R}_{\delta,\mathrm{aff}} \subset \mathcal{R}$ sufficiently far from the boundary it acts as an affine function with differential 
\begin{align} \label{eq:differential.on.rectangles}
  AJA^{-1}
  =
  \begin{pmatrix}
    0&-L_1/L_2\\
    L_2/L_1&0
  \end{pmatrix}.
\end{align}

\end{lemma}

\begin{proof}
For every $\delta \in (0,1/8)$ let $N_\delta \in 2 \NN$ be an integer satisfying $\delta^{-1}/2 < N_\delta < 2\delta^{-1}$ and define the set
\begin{align}
  D_\delta
  :=
  \left\{
    (x,y)\in\R^2:
    x^{N_\delta}+y^{N_\delta}<(1-\delta)^{N_\delta}
  \right\}
  \subset \mathcal{Q}.
\end{align}

\emph{Step 1}. Let us preliminarily verify that there exist constants $0<c_\star <C_\star$ such that, for every $\delta$ sufficiently small, the following inclusions hold:
\begin{align} \label{eq:containment}
  \overline{(1-C_\star\delta)\mathcal{Q}} \subset D_\delta,
  \qquad
  \overline{D_\delta } \subset (1-c_\star\delta)\mathcal{Q}.
\end{align}
For the first inclusion notice that we have $x^{N_\delta}+y^{N_\delta} \leq 2 (1-C_\star \delta)^{N_\delta}$ for every $(x,y) \in \overline{(1-C_\star\delta)\mathcal{Q}}$. Thus it is sufficient to choose $C_\star \gg 1$ so large that
\begin{align} \label{eq:inequality.Cstar.1}
    2 (1-C_\star \delta)^{N_\delta} < (1-\delta)^{N_\delta},
\end{align}
which is always possible for small $\delta \ll 1$ since by taking logarithms and Taylor expanding
\begin{align}
    \log 2 + N_\delta \log (1-C_\star \delta) 
    <
    \log 2 - \frac12 C_\star \delta N_\delta 
    < 
    \log 2 - \frac14 C_\star,
\end{align}
while on the other side of \eqref{eq:inequality.Cstar.1}
\begin{align}
  -4
    <
    -2N_\delta \delta
    <N_\delta \log(1-\delta)  .
\end{align}

The second inclusion in \eqref{eq:containment} follows by a similar argument, after choosing $c_\star \ll 1$ such that
\begin{align}
(1-\delta)^{N_\delta} < (1-c_\star \delta)^{N_\delta}.
\end{align}

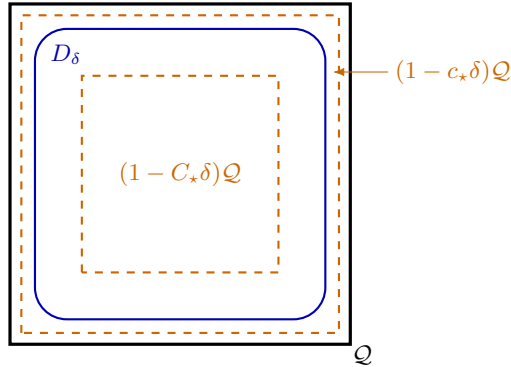
\begin{figure}[h]
  \centering
  \begin{tikzpicture}[x=1cm,y=1cm,font=\small]
    \draw[very thick] (-2.25,-2.25) rectangle (2.25,2.25);
    \node[anchor=north west] at (2.15,-2.15) {$\mathcal{Q}$};

    \draw[blue!70!black,thick,rounded corners=12pt]
      (-1.92,-1.92) rectangle (1.92,1.92);
    \node[blue!70!black,anchor=north west] at (-1.84,1.84)
      {$D_\delta$};

    \draw[orange!80!black,dashed,thick]
      (-1.30,-1.30) rectangle (1.30,1.30);
    \node[orange!80!black,align=center] at (0,0)
      {$(1-C_\star\delta)\mathcal{Q}$};

\draw[orange!80!black,dashed,thick]
      (-2.1,-2.1) rectangle (2.1,2.1);
    \draw[-{Latex[length=1.5mm]},orange!80!black]
      (2.75,1.35) -- (2.02,1.35);
    \node[orange!80!black,anchor=west,align=left] at (2.72,1.35)
      {$(1-c_\star\delta)\mathcal{Q}$};
  \end{tikzpicture}
  \caption{The map $T_\delta$ is an exact $90^\circ$ rotation inside the small square $(1-C_\star\delta)\mathcal{Q}$, the identity outside of the large square $(1-c_\star\delta)\mathcal{Q}$, and smooth in between.}
\end{figure}

\emph{Step 2}.
Next, we want to parametrize points $(x,y) \in \partial D_\delta$ in polar coordinates.
For every angle $\phi \in [0,2\pi)$ between the $x$-axis and the vector $(x,y)$, we have that $(x,y) =: \rho_\delta(\phi) (\cos(\phi),\sin(\phi)) \in \partial D_\delta$ if and only if
\begin{align}
    \rho_\delta(\phi)^{N_\delta}
    \left( \cos(\phi)^{N_\delta}+\sin(\phi)^{N_\delta} \right)=
    x^{N_\delta} + y^{N_\delta} 
    = 
    (1-\delta)^{N_\delta} .
\end{align}

For every $\delta$ as above, the function $\phi \mapsto \rho_\delta(\phi)$ is smooth. Let us check that it satisfies:
\begin{align} \label{eq:derivatives.rho.delta}
    \| \rho_\delta \|_{C^k} \lesssim 
    \begin{cases}
        1,
        \quad
        &k =0,1,
        \\
        \delta^{-1},
        \quad
        &k=2.
    \end{cases}
\end{align}

The case $k=0$ follows by the inequality between means. Indeed,
\begin{align} \label{eq:bound.rhodelta.from.below}
\left( \frac{\cos(\phi)^{N_\delta}+\sin(\phi)^{N_\delta}}{2} \right)^{\frac{1}{N_\delta}}   
\geq
\left( \frac{\cos(\phi)^2+\sin(\phi)^2}{2} \right)^{\frac{1}2}
= \frac{1}{\sqrt{2}}
\end{align}
and therefore we have
\begin{align} \label{eq:bound.cos+sin-1}
    \left( \cos(\phi)^{N_\delta}+\sin(\phi)^{N_\delta} \right)^{-\frac{1}{N_\delta}}
\lesssim 1.
\end{align}

Let us compute the first derivative
\begin{align}
\rho_\delta'(\phi)
=
(\delta-1) \left( \cos(\phi)^{N_\delta}+\sin(\phi)^{N_\delta} \right)^{-\frac{1}{N_\delta}-1} \cos(\phi)\sin(\phi) \left( \sin(\phi)^{N_\delta-2}-\cos(\phi)^{N_\delta-2}\right).
\end{align}
Then by triangular inequality and \eqref{eq:bound.cos+sin-1} we have
\begin{align}
    |\rho_\delta'(\phi)|
    &\lesssim
    \left( \cos(\phi)^{N_\delta}+\sin(\phi)^{N_\delta} \right)^{-\frac{1}{N_\delta}-1} |\sin(\phi)| \sin(\phi)^{N_\delta-2}
    \\
    &\quad+
    \left( \cos(\phi)^{N_\delta}+\sin(\phi)^{N_\delta} \right)^{-\frac{1}{N_\delta}-1} |\cos(\phi)|  \cos(\phi)^{N_\delta-2}
    \\
    &\lesssim
    \left( \cos(\phi)^{N_\delta}+\sin(\phi)^{N_\delta} \right)^{\frac{1}{N_\delta}-1} |\sin(\phi)| \sin(\phi)^{N_\delta-2}
    \\
    &\quad+
    \left( \cos(\phi)^{N_\delta}+\sin(\phi)^{N_\delta} \right)^{\frac{1}{N_\delta}-1} |\cos(\phi)|  \cos(\phi)^{N_\delta-2}
    \\
    &\lesssim
    |\sin(\phi)|^{1-N_\delta} |\sin(\phi)| \sin(\phi)^{N_\delta-2}
    +
    |\cos(\phi)|^{1-N_\delta} |\cos(\phi)|  \cos(\phi)^{N_\delta-2}
    \lesssim 1,
\end{align}
where in the last line we have used $\cos(\phi)^{N_\delta}+\sin(\phi)^{N_\delta} \geq \min\{\cos(\phi)^{N_\delta},\sin(\phi)^{N_\delta}\}$ since $N_\delta$ is an even integer.

Similarly, the second derivative is given by
\begin{align}
\rho_\delta''(\phi)
&=
(1-\delta) \left(1+N_\delta\right) 
\left( \cos(\phi)^{N_\delta}+\sin(\phi)^{N_\delta} \right)^{-\frac{1}{N_\delta}-2} \cos(\phi)^2\sin(\phi)^2 \left( \sin(\phi)^{N_\delta-2}-\cos(\phi)^{N_\delta-2}\right)^2
\\
&\quad+
(\delta-1) \left( \cos(\phi)^{N_\delta}+\sin(\phi)^{N_\delta} \right)^{-\frac{1}{N_\delta}-1} (\cos(\phi)^2-\sin(\phi)^2) \left( \sin(\phi)^{N_\delta-2}-\cos(\phi)^{N_\delta-2}\right)
\\
&\quad+
(\delta-1) \left(N_\delta-2\right)\left( \cos(\phi)^{N_\delta}+\sin(\phi)^{N_\delta} \right)^{-\frac{1}{N_\delta}-1} \cos(\phi)^2\sin(\phi)^2 \left( \sin(\phi)^{N_\delta-4}+\cos(\phi)^{N_\delta-4}\right).
\end{align}

Then by triangular inequality and arguing as above we have
\begin{align}
|\rho_\delta''(\phi)|
&\lesssim
N_\delta 
\left( \cos(\phi)^{N_\delta}+\sin(\phi)^{N_\delta} \right)^{\frac{2}{N_\delta}-2} \sin(\phi)^2 \sin(\phi)^{2N_\delta-4}
\\
&\quad+
N_\delta 
\left( \cos(\phi)^{N_\delta}+\sin(\phi)^{N_\delta} \right)^{\frac{2}{N_\delta}-2} \cos(\phi)^2 \cos(\phi)^{2N_\delta-4}
\\
&\quad+
\left( \cos(\phi)^{N_\delta}+\sin(\phi)^{N_\delta} \right)^{\frac{2}{N_\delta}-1} \sin(\phi)^{N_\delta-2}
\\
&\quad+
\left( \cos(\phi)^{N_\delta}+\sin(\phi)^{N_\delta} \right)^{\frac{2}{N_\delta}-1} \cos(\phi)^{N_\delta-2}
\\
&\quad+
N_\delta
\left( \cos(\phi)^{N_\delta}+\sin(\phi)^{N_\delta} \right)^{\frac{2}{N_\delta}-1} \sin(\phi)^2 \sin(\phi)^{N_\delta-4}
\\
&\quad+
N_\delta
\left( \cos(\phi)^{N_\delta}+\sin(\phi)^{N_\delta} \right)^{\frac{2}{N_\delta}-1} \cos(\phi)^2 \cos(\phi)^{N_\delta-4}
\\
&\lesssim N_\delta
\lesssim
\delta^{-1}.
\end{align}

\emph{Step 3}.
We now parametrize $D_\delta$ by \emph{swept area}. We look for a map $\varphi_\delta : [0,1) \to [0,2\pi)$ such that, for every $\vartheta \in [0,1)$ the fraction of area of $D_\delta$ that, in polar coordinates, is enclosed in the interval $\phi \in [\varphi_\delta(0),\varphi_\delta(\vartheta))$ is exactly equal to $\vartheta$, in formulae:
\begin{align}
\int_{\varphi_\delta(0)}^{\varphi_\delta(\vartheta)} 
\int_0^{\rho_\delta(\phi)} \rho \,d\rho \,d\phi
    =
    \frac12 \int_{\varphi_\delta(0)}^{\varphi_\delta(\vartheta)} 
\rho_\delta(\phi)^2d\phi
:=
    \vartheta \mathscr{L}^2 (D_\delta).
\end{align}

Since $\rho_\delta$ is bounded away from zero for every $\delta \ll 1$, such $\varphi_\delta$ can be simply obtained by inverting the function:
\begin{align}
 \vartheta(\varphi_\delta) 
 :=
 \frac{\int_0^{\varphi_\delta}\rho_\delta(\phi)^2\,d\phi}
       {\int_0^{2\pi}\rho_\delta(\phi)^2\,d\phi},
\end{align}
and by \eqref{eq:derivatives.rho.delta} and \eqref{eq:bound.rhodelta.from.below} and the formula for the derivative of the inverse, the function $\vartheta \mapsto \varphi_\delta(\vartheta)$ has derivatives of order $k \in \{0,1,2\}$ bounded by an unimportant constant $\| \varphi_\delta \|_{C^k} \lesssim 1$.

Hence the boundary of $D_\delta$ can be parametrized as follows
\begin{align}
    p_\delta(\vartheta)
    &:=
    \rho_\delta(\varphi_\delta(\vartheta)) 
    ( \cos(\varphi_\delta(\vartheta)),\sin(\varphi_\delta(\vartheta)) ),
    \quad
    \vartheta \in [0,1).
\end{align}
Let $z_\mathcal{Q} := (0,0)$ denote the center of the square $\mathcal{Q}$.
Every point $z \in D_\delta \setminus \{z_\mathcal{Q}\}$ is uniquely parametrized in these coordinates by
\begin{align} \label{eq:new.coordinates}
    Z(r,\vartheta) &:= z = r p_\delta(\vartheta),
    \quad
    r \in (0,1),
    \vartheta \in [0,1).
\end{align}
The determinant of the Jacobian of the parametrization is
\begin{align}
\mathrm{det} (DZ(r,\vartheta)) &=
\mathrm{det} 
\begin{pmatrix} 
\rho_\delta(\varphi_\delta(\vartheta)) 
\cos(\varphi_\delta(\vartheta)) 
& 
r (\rho_\delta(\varphi_\delta(\vartheta)))'\cos(\varphi_\delta(\vartheta))
-
r \rho_\delta(\varphi_\delta(\vartheta)) 
\sin(\varphi_\delta(\vartheta)) \varphi_\delta'(\vartheta)
  \\
\rho_\delta(\varphi_\delta(\vartheta)) 
\sin(\varphi_\delta(\vartheta)) 
& 
r (\rho_\delta(\varphi_\delta(\vartheta)))'\sin(\varphi_\delta(\vartheta))
+
r \rho_\delta(\varphi_\delta(\vartheta)) 
\cos(\varphi_\delta(\vartheta)) \varphi_\delta'(\vartheta) 
  \end{pmatrix}
  \\
  &=
  r \rho_\delta(\varphi_\delta(\vartheta))^2 \varphi_\delta'(\vartheta)
  \\
  &=
  r \int_0^{2\pi} \rho_\delta(\phi)^2 d\phi ,
\end{align}
where the last equality descends from the formula for the derivative of the inverse, and therefore
\begin{align}
    dz = \mathrm{det} (DZ(r,\vartheta)) \,dr \, d\vartheta
    =r \,dr \, d\vartheta  \int_0^{2\pi} \rho_\delta(\phi)^2 d\phi 
    .
    \label{eq:measure.preserving}
\end{align}

\emph{Step 4}.
We are now ready to define the diffeomorphism $T_\delta$.
Fix a smooth cutoff function $\chi_\delta \in C^\infty$ such that 
\begin{align}
\chi_\delta(r) \equiv \frac14
  \quad\text{for }r\leq1-2\delta,
  \quad
\chi_\delta(r) \equiv 0
  \quad\text{for }r\geq1-\delta,
\end{align}
and
\begin{align}
    \|\chi_\delta\|_{C^k}\lesssim \delta^{-k},
    \quad
    \mbox{ for }k \in \{0,1,2\}.
\end{align}

Set $T_\delta(z_\mathcal{Q})=z_\mathcal{Q}$ and, in the coordinates $(r,\vartheta)$ above,
\begin{align} \label{eq:definition.Tdelta}
    T_\delta (r,\vartheta) := (r,\vartheta+\chi_\delta(r)).
\end{align}
Up to possibly increasing the value of $C_\star$ and restricting to smaller values of $\delta$ (i.e. decreasing $\delta_\star$), arguing similarly to Step 1 we can suppose that the containment $(1-C_\star \delta) \mathcal{Q} \subset (1-2\delta)D_\delta$ holds.
Therefore by fourfold symmetry and invariance of the swept area by rotations we have for every point $z = rp_\delta(\vartheta)$ in these coordinates
\begin{align}
   T_\delta (rp_\delta(\vartheta))& = rp_\delta(\vartheta+1/4) = J(rp_\delta(\vartheta)),
   &&\forall z \in (1-C_\star \delta) \mathcal{Q},
   \\
   T_\delta (rp_\delta(\vartheta)) &= rp_\delta(\vartheta),
   &&\forall z \notin (1-c_\star \delta) \mathcal{Q}.
\end{align}
Moreover, $T_\delta$ preserves the measure $\mathrm{det} (DZ(r,\vartheta)) \,dr \, d\vartheta$ and thus it preserves $dz$ by \eqref{eq:measure.preserving}.

For the derivative bounds on $T_\delta$, notice that by construction the differential of $T_\delta$ is constant in the regions $\mathrm{dist}(z,\partial \mathcal{Q})<c_\star \delta$ and $\mathrm{dist}(z,\partial \mathcal{Q})>C_\star \delta$, thus the bounds hold there.
In the transition region, we argue as follows. 
First, by differentiating the expression for $T_\delta$ \eqref{eq:definition.Tdelta} and $Z$ \eqref{eq:new.coordinates} with respect to the coordinates $(r,\vartheta)$, one has that 
\begin{align}
    \| D^k_{r,\vartheta} T_\delta \|_{L^\infty} + \| D^k_{r,\vartheta} Z \|_{L^\infty}
    \lesssim
    \delta^{-k}.
\end{align}
Second, the inverse map to \eqref{eq:new.coordinates} $Z^{-1} : z \mapsto (r,\vartheta)$ satisfies $Z \circ Z^{-1} (z) \equiv z $ and by chain rule we have that
\begin{align}
 DZ^{-1} &= ((DZ) \circ Z^{-1}  )^{-1},
 \\
 D^2Z^{-1} &=-((DZ) \circ Z^{-1})^{-1}((D^2Z) \circ Z^{-1}) DZ^{-1} DZ^{-1},
\end{align}
and therefore, using the lower bound on $\det (DZ)$ from \eqref{eq:measure.preserving} we deduce
\begin{align}
    \| D^k Z^{-1}\|_{L^\infty} \lesssim \delta^{-3k}.
\end{align}
Applying again the chain rule to $Z \circ T_\delta \circ Z^{-1}$ we finally obtain
\begin{align}
    \| T_\delta\|_{C^k} \lesssim \delta^{-m_k},
    \quad
    \mbox{ for some } m_k< \infty.
\end{align}
Replacing $\chi_\delta$ by $-\chi_\delta$ above gives the same bounds for the inverse $T_\delta^{-1}$ and proves \eqref{eq:bound.Ck.Tdelta}.  
Extending $T_\delta$ as the identity outside $D_\delta$ we obtain our desired map and the proof of the first part of the lemma is complete.

\emph{Step 5}.
Finally, we consider the map $T_{\delta,\mathcal{R}}$. Let us denote $T_\delta(z) = (T_\delta^1(z),T_\delta^2(z))$ the components of $T_\delta$, and for $z=(z^1,z^2) \in \mathcal{R}$
\begin{align}
    T_{\delta,\mathcal{R}}(z^1,z^2)
    =
    \begin{pmatrix}
    T_{\delta,\mathcal{R}}^1(z^1,z^2)
    \\
    T_{\delta,\mathcal{R}}^2(z^1,z^2)
    \end{pmatrix}
    =
    \begin{pmatrix}
        z_0^1 + L_1 T^1_\delta \left(\frac{z^1-z_0^1}{L_1} , \frac{z^2-z_0^2}{L_2} \right)
        \\
         z_0^2 + L_2 T^2_\delta \left(\frac{z^1-z_0^1}{L_1} , \frac{z^2-z_0^2}{L_2} \right)
    \end{pmatrix}.
\end{align}
Then for every component $i \in \{1,2\}$ and multi index $a=(a_1,a_2)$ with $|a| = k \in \{1,2\}$ it holds
\begin{align}
    \| \partial^a T_{\delta,\mathcal{R}}^i \|_{L^\infty}
    \lesssim
    \frac{L_i}{L_1^{a_1}  L_2^{a_2}} \| T_\delta\|_{C^k}
    \lesssim
    \frac{L_i}{\min\{L_1,L_2\}^k} \delta^{-m_k}.
\end{align}
A similar argument applies to the inverse map.
The other claims on $T_{\delta,\mathcal{R}}$ follow immediately by construction and basic linear algebra.
\end{proof}

\subsection{Composing rectangular quarter turns}
\label{ssec:composition}
In this subsection we want to compose rectangular quarter turns from previous \autoref{ssec:quarter-turn} so that the square $\mathcal{Q}$ is ``expanded'' in the vertical direction and ``contracted'' in the horizontal direction, at least on a large portion of $\mathcal{Q}$, and then ``shuffled'' at small scales.

To do this, let us fix parameters $W,H,V$ as follows:
\begin{align}  \label{eq:definitionWHV}
    W<VH<1,
  \qquad
  H<W,
  \qquad
  2HV^\alpha < W^{1+\alpha}.
\end{align}
Assume moreover that $V,\frac{1}{W},\frac{1}{H} \geq 2$ are integers. Notice that this can be also done under the assumption $\alpha \in (0,1)$, for instance by fixing rational numbers $0<\varepsilon_1<\varepsilon_2<\frac{1-\alpha}{1+\alpha}$ and an integer $V \gg 2$ such that 
\begin{align}
\frac1W := V^{\varepsilon_2},
\quad
\frac1H:=V^{1+\varepsilon_1}    
\end{align}
are also integers $\geq 2$ satisfying \eqref{eq:definitionWHV}.

Next, divide the square $\mathcal{Q}$ into $\frac{2}{W} \times \frac{2}{H}$ smaller rectangles $\{ \mathcal{R}^{i,j}\}_{\substack{i=1,\dots,2/W ; \,j=1,\dots,2/H}}$ each having width $W$ and height $H$.

In every rectangle $\mathcal{R}^{i,j}$, apply the map $T_{\delta,\mathcal{R}^{i,j}}$ from \autoref{lem:quarter.turn}. Since each $T_{\delta,\mathcal{R}^{i,j}}$ coincides with the identity in a neighborhood of $\partial \mathcal{R}^{i,j}$, the maps $\{ T_{\delta,\mathcal{R}^{i,j}} \}_{i,j}$ can be glued together to form a smooth, measure preserving diffeomorphism of the whole square $\mathcal{Q}$.

Then, further subdivide each rectangle $\mathcal{R}^{i,j}$ into $V$ vertical columns $\{ \mathcal{R}^{i,j,k}\}_{k=1,\dots,V}$ each of width $w := W/V$, and apply the inverse map  $T_{\delta^{1/2},\mathcal{R}^{i,j,k}}^{-1}$ on every $\mathcal{R}^{i,j,k}$.
Again, by construction the maps $\{T_{\delta^{1/2},\mathcal{R}^{i,j,k}}^{-1}\}_{i,j,k}$ can be glued together to form a smooth, measure preserving diffeomorphism of $\mathcal{Q}$. Here we are implicitly assuming that $\delta<\delta^{1/2}<\delta_\star$, so that \autoref{lem:quarter.turn} applies.

In formulae, we can define the transformation $B_\delta : \mathcal{Q} \to \mathcal{Q}$ above as follows:
\begin{align}
    B_\delta(z) := T_{\delta^{1/2},\mathcal{R}^{i,j,k}}^{-1} ( T_{\delta,\mathcal{R}^{i,j}}(z)),
    \quad
    \mbox{if } z \in T_{\delta,\mathcal{R}^{i,j}}^{-1}(\mathcal{R}^{i,j,k}).
\end{align}

Moreover, recall from \autoref{lem:quarter.turn} and \eqref{eq:differential.on.rectangles} that for each $\delta,\delta^{1/2} \in (0,\delta_\star)$ and $i,j,k$ there exist regions $\mathcal{R}^{i,j}_{\delta,\mathrm{aff}} \subset \mathcal{R}^{i,j}$ and $\mathcal{R}^{i,j,k}_{\delta^{1/2},\mathrm{aff}} \subset \mathcal{R}^{i,j,k}$ such that the derivatives of $T_{\delta,\mathcal{R}^{i,j}}$, $T_{\delta^{1/2},\mathcal{R}^{i,j,k}}^{-1}$ therein are given by
\begin{align}
    D T_{\delta,\mathcal{R}^{i,j}} |_{\mathcal{R}^{i,j}_{\delta,\mathrm{aff}}}
    \equiv
    \begin{pmatrix}
    0&-W/H\\
    H/W&0
  \end{pmatrix},
  \quad
   D T_{\delta,\mathcal{R}^{i,j,k}}^{-1} |_{\mathcal{R}^{i,j,k}_{\delta^{1/2},\mathrm{aff}}}
    \equiv
    \begin{pmatrix}
    0&w/H\\
    -H/w&0
  \end{pmatrix}.
\end{align}

We say that $z \in \mathcal{P}_{\delta} \subset \mathcal{Q}$ if both $z \in \mathcal{R}^{i,j}_{\delta,\mathrm{aff}}$ and $T_{\delta,\mathcal{R}^{i,j}}(z) \in \mathcal{R}^{i,j,k}_{\delta^{1/2},\mathrm{aff}}$ for some $i,j,k$.
The derivative of $B_\delta$ on $\mathcal{P}_\delta$ is given by the product matrix
\begin{align} \label{eq:DBdelta}
    D B_\delta|_{\mathcal{P}_\delta} \equiv 
    \begin{pmatrix}
    0&w/H\\
    -H/w&0
  \end{pmatrix}\begin{pmatrix}
    0&-W/H\\
    H/W&0
  \end{pmatrix}
  =
  \begin{pmatrix}
    w/W&0\\
    0&W/w
  \end{pmatrix}
  =
  \begin{pmatrix}
    V^{-1}&0\\0&V
  \end{pmatrix}.
\end{align}

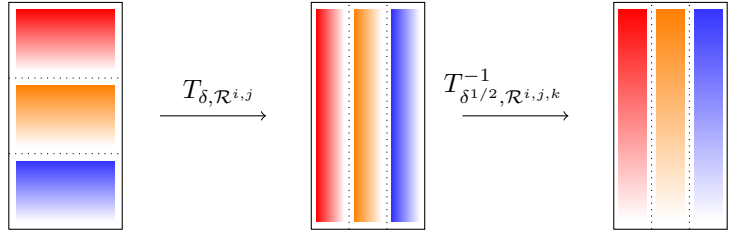
\begin{figure}[h]
    \centering

    \begin{tikzpicture}

\shade[top color=red,bottom color=white] (0.1,2.1) rectangle (1.4,2.9);
\shade[top color=orange,bottom color=white] (0.1,1.1) rectangle (1.4,1.9);
\shade[top color=blue!80,bottom color=white] (0.1,0.1) rectangle (1.4,0.9);

\draw (0,0) -- (0,3) -- (1.5,3) -- (1.5,0) -- (0,0) ;

\draw [dotted] (0, 1) -- (1.5, 1);
\draw [dotted] (0, 2) -- (1.5, 2);

\draw [-to] (2,1.5) -- (3.4,1.5)  node[above left]{$T_{\delta,\mathcal{R}^{i,j}}$};

\begin{scope}[shift={(4,0)}]
\shade[left color=red,right color=white] (0.065,0.1) rectangle (0.435,2.9);
\shade[left color=orange,right color=white] (0.565,0.1) rectangle (0.935,2.9);
\shade[left color=blue!80,right color=white] (1.065,0.1) rectangle (1.435,2.9);

\draw (0,0) -- (0,3) -- (1.5,3) -- (1.5,0) -- (0,0) ;

\draw [dotted] (0.5, 0) -- (0.5, 3);
\draw [dotted] (1, 0) -- (1, 3);

\draw [-to] (2,1.5) -- (3.4,1.5)  node[above left]{$T_{\delta^{1/2},\mathcal{R}^{i,j,k}}^{-1} \!$};
\end{scope}

\begin{scope}[shift={(4,0)}]
    \begin{scope}[shift={(4,0)}]
\shade[top color=red,bottom color=white] (0.065,0.1) rectangle (0.435,2.9);
\shade[top color=orange,bottom color=white] (0.565,0.1) rectangle (0.935,2.9);
\shade[top color=blue!80,bottom color=white] (1.065,0.1) rectangle (1.435,2.9);

\draw (0,0) -- (0,3) -- (1.5,3) -- (1.5,0) -- (0,0) ;

\draw [dotted] (0.5, 0) -- (0.5, 3);
\draw [dotted] (1, 0) -- (1, 3);

\end{scope}
\end{scope}

\end{tikzpicture}
\caption{Action of $B_\delta$ on a representative rectangle $\mathcal{R}^{i,j}$. Colored subrectangles belong to the set $\mathcal{P}_\delta$, on which $B_\delta$ is affine, contracts width by a factor $V^{-1}$ and expands height by a factor $V$. Shades of colors indicate the orientation of the rectangles at each step.}
\end{figure}

Next, we want to determine some properties of the set $\mathcal{P}_\delta$.
Denote $C_\star$ the constant from \autoref{lem:quarter.turn} and introduce
\begin{align}
    \lambda_\delta:=1-C_\star \delta^{1/2}.
\end{align}
Inside every column $\mathcal{R}^{i,j,k}$ take a smaller concentric rectangle $\underline{\mathcal{R}}^{i,j,k} \subset \mathcal{R}^{i,j,k}$ of width $\lambda_\delta w$ and height $\lambda_\delta H$. Then by construction it holds
\begin{align} \label{eq:underlineRaff}
\underline{\mathcal{R}}^{i,j,k} \subset \mathcal{R}^{i,j,k}_{\delta^{1/2},\mathrm{aff}} \subset  \mathcal{R}^{i,j,k}   . 
\end{align}
If we additionally require $\delta$ to be small enough so that
\begin{align}
    (1-\lambda_\delta) w &> 2C_\star \delta W,
    \\
    (1-\lambda_\delta) H &> 2C_\star \delta H,
\end{align}
then
\begin{align} \label{eq:underlineR}
 \underline{\mathcal{R}}^{i,j,k} \subset T_{\delta,\mathcal{R}^{i,j}} (\mathcal{R}^{i,j}_{\delta,\mathrm{aff}}).
\end{align}
By construction, the union of the rectangles $\underline{\mathcal{R}}^{i,j,k}$ is a cartesian product
\begin{align}
   \bigcup_{i,j,k}  \underline{\mathcal{R}}^{i,j,k} =: I_x \times I_y,
\end{align}
where $I_x$ is the union of $2/w$ disjoint intervals, each of length $\lambda_\delta w$, and $I_y$ is the union of $2/H$ disjoint intervals, each of length $\lambda_\delta H$. 

By applying $T_{\delta,\mathcal{R}^{i,j}}^{-1}$ to \eqref{eq:underlineR} we have by \eqref{eq:underlineRaff} and the very definition of $\mathcal{P}_\delta$
\begin{align}
 T_{\delta,\mathcal{R}^{i,j}}^{-1}(\underline{\mathcal{R}}^{i,j,k}) \subset    \mathcal{P}_\delta,
\end{align}
and thus
\begin{align} \label{eq:definition.affinebulk}
\bigcup_{i,j,k}  T_{\delta,\mathcal{R}^{i,j}}^{-1}(\underline{\mathcal{R}}^{i,j,k}) =: U_x \times U_y\subset    \mathcal{P}_\delta,  
\end{align}
where now $U_x$ is the union of $2/W$ disjoint intervals, each of length $\lambda_\delta W$, and $U_y$ is the union of $2V/H$ disjoint intervals, each of length $\lambda_\delta H/V$.

By \eqref{eq:bound.Ck.Tdelta} from \autoref{lem:quarter.turn} and chain rule, it holds for every fixed $k \in \{1,2\}$
\begin{align} \label{eq:bounds.B.delta}
    \| B_\delta \|_{C^k} + \| B_\delta^{-1} \|_{C^k} \lesssim 
    \delta^{-M_k},
\end{align}
with $M_k$ depending only on $m_1,m_2$ and the implicit constant depending only on $W,H,V$ and the implicit constant of \eqref{eq:bound.Ck.Tdelta}.

In order to simplify formulae below and allow for a more convenient treatment of the construction in the following, let us introduce the following nomenclature:
\begin{definition}
Let $\delta \in (0,\delta_\star)$ be given.
Define the \emph{affine bulk} of the map $B_\delta$ as the left-hand side of \eqref{eq:definition.affinebulk} above, and the \emph{affine generation} as the  collection of rectangles in the affine bulk
\begin{align} 
\mathscr{P} := \{ \mathcal{R} \,:\, \exists i,j,k \mbox{ such that } \mathcal{R} = T_{\delta,\mathcal{R}^{i,j}}^{-1}(\underline{\mathcal{R}}^{i,j,k}) \}.  
\end{align}
By construction, for every rectangle $\mathcal{R}$ in the affine generation it holds $\mathcal{R} \subset \mathcal{P}_\delta$ and thus $B_\delta$ acts affinely on the affine bulk.


\end{definition}


\subsection{Choice of parameters and recursion}
\label{ssec:recursion}

Given $0<\delta\ll 1$ and a rectangle $\mathcal{R} \subset \mathcal{Q}$ of width $\ell_1$ and height $\ell_2$, we define $B_{\delta,\mathcal{R}} : \mathcal{R} \to \mathcal{R}$ as the affine conjugate of the map $B_{\delta}$ similar to that introduced in \autoref{lem:quarter.turn}. By \eqref{eq:bounds.B.delta} and arguing similarly to Step 5 in the proof of \autoref{lem:quarter.turn} we have the following bounds on $B_{\delta,\mathcal{R}}$, for every component $i \in \{1,2\}$ and multi index $a=(a_1,a_2)$ with $|a| = k \in \{1,2\}$ 
\begin{align} \label{eq:bounds.B.delta.R}
    \| \partial^a ( B_{\delta,\mathcal{R}})^i \|_{L^\infty(\mathcal{R})} 
    + 
    \| \partial^a ( B_{\delta,\mathcal{R}}^{-1})^i  \|_{L^\infty(\mathcal{R})} \lesssim \frac{\ell_i}{\ell_1^{a_1}\ell_2^{a_2}} 
    \delta^{-M_k}.
\end{align}

In this subsection we shall recursively iterate applications of such transformations, with suitable choices of the parameter $\delta$ and rectangles $\mathcal{R}$.

Recall $\mathcal{R}_0 = R_s \times R_h$ from \autoref{ssec:Psi_0} and define for notational convenience $\mathscr{P}_{-1} := \{\mathcal{R}_0\}$.

At each step $q \in \NN$ of the recursion, we define a map $\beta_q$ as the composition of all maps $B_{\delta_q,\mathcal{R}}$ with rectangles $\mathcal{R} \in \mathscr{P}_{q-1}$ in the $(q-1)$-th affine generation, and with parameter
\begin{align} \label{eq:definition.delta.q}
  \delta_q:=(q+q_0)^{-p},
  \quad
  \mbox{ for some }q_0,p>2.
\end{align}
We assume throughout the iteration that $q_0 \gg 1$ is large enough that for every $q \in \NN$
\begin{align}
    \delta_q<\delta_q^{1/2}<\delta_\star,
    \quad
    C_\star q_0^{-p/2}\leq 1/2,
    \quad
    1/2\leq\lambda_{\delta_q}<1.
\end{align}

Since each $B_{\delta_q,\mathcal{R}}$ acts as the identity in a neighborhood of $\partial\mathcal{R}$,  the map $\beta_q$ is a smooth measure-preserving diffeomorphism of $\mathcal{R}_0$ that satisfies the same bounds in $C^k$ as each $B_{\delta_q,\mathcal{R}}$ given by \eqref{eq:bounds.B.delta.R} above. 

Then we denote $\mathscr{P}_q$ the $q$-th affine generation of rectangles in the affine bulk, defined as the collection of all rectangles in the affine generation of the map $B_{\delta_q,\mathcal{R}}$ for some $\mathcal{R} \in \mathscr{P}_{q-1}$.
Notice that the affine bulks are nested:
\begin{align} \label{eq:nested.bulks}
    \mathcal{P}_q := \bigcup_{\mathcal{R} \in \mathscr{P}_q} \mathcal{R} \subset \mathcal{P}_{q-1}.
\end{align}

Finally, we recall $\Psi_0 : \TT \times I \to \mathcal{A}$ from \autoref{lem:Psi_0.diffeo} and recursively define a diffeomorphism $\Psi_q : \TT \times I \to \mathcal{A}$ as
\begin{align}
  \Psi_{q+1} := \Psi_q \circ \beta_q = \Psi_0 \circ \beta_0 \circ \dots \circ \beta_q.  
\end{align}
Notice that $\Psi_{q+1}$ is affine on each rectangle $\mathcal{R} \in \mathscr{P}_q$ by construction.

\begin{lemma} \label{lem:sizes}
Denote $s_{-1}:= |R_s|$ and $h_{-1}:= |R_h|$, and
\begin{align}
\Lambda_q:=\prod_{j=0}^{q}\lambda_{\delta_j}.
\end{align}
Then for every $\mathcal{R} \in \mathscr{P}_q$ it holds $\mathcal{R} = R^q_s \times R^q_h$ for some intervals $R^q_s,R^q_h$ of length
 \begin{align}
|R^q_s| =: s_q =  s_{-1} \Lambda_q \left(\frac{W}{2}\right)^{q+1}
,
\quad
|R^q_h| =: h_q = h_{-1} \Lambda_q \left(\frac{H}{2V}\right)^{q+1},
 \end{align}
and the image $\Psi_{q+1}(\mathcal{R}) = A^q_x \times A^q_y$ for some intervals $A^q_x,A^q_y$ of length
 \begin{align}
|A^q_x| =: x_q = s_{-1} \Lambda_q \left(\frac{W}{2V}\right)^{q+1}
,
\quad
|A^q_y| =: y_q = h_{-1}  \Lambda_q \left(\frac{H}{2}\right)^{q+1}.
 \end{align} 
\end{lemma}

\begin{proof}
The fact that $\mathcal{R}$ is a rectangle for every $q$ follows by the definition of $\mathscr{P}_q$. The width and height of $\mathcal{R}$ can be recursively computed by 
\begin{align}
    |R^q_s| &= \lambda_{\delta_q} \left(\frac{W}{2}\right) s_{q-1} = s_{-1} \Lambda_q \left(\frac{W}{2}\right)^{q+1},
    \\
    |R^q_h| &= \lambda_{\delta_q} \frac{H}{2V} h_{q-1} = h_{-1} \Lambda_q \left(\frac{H}{2V}\right)^{q+1}.
\end{align}
By \eqref{eq:Psi.affine} the map $\Psi_{q+1}$ is affine on $\mathcal{R}$ and the image $\Psi_{q+1}(\mathcal{R}) = A^q_x \times A^q_y$ is a rectangle, with sides parallel to the coordinate axes by \eqref{eq:extra.assumption.curve}.  
By recalling \eqref{eq:DBdelta}, each $\beta_j$ contracts width and expands height by a factor $V$, for $j =0,\dots,q$. Thus 
\begin{align}
    |A^q_x| &= 
    s_{-1}  \Lambda_q \left(\frac{W}{2V}\right)^{q+1},
    \\
    |A^q_y| &= 
    h_{-1}  \Lambda_q \left(\frac{H}{2}\right)^{q+1}.
\end{align}
\end{proof}

\begin{lemma}
\label{lem:homeo}
The maps $\Psi_q,\Psi_q^{-1}$ converge uniformly as $q \to \infty$ to (mutually
inverse) measure-preserving homeomorphisms
\begin{align}
  \Psi:\TT \times I\to\mathcal A,
  \qquad
  \Psi^{-1}:\mathcal A\to \TT \times I.  
\end{align}
\end{lemma}

\begin{proof}
Let $z \in \TT \times I$ and recall $\Psi_{q+1}(z)=\Psi_q(\beta_q(z))$.
If $\beta_q(z) \neq z$, then by definition $\beta_q(z) = B_{\delta_q,\mathcal{R}}(z)$ for some rectangle $\mathcal{R} \in \mathscr{P}_{q-1}$ in the $(q-1)$-th affine generation containing $z$.
Since $B_{\delta_q,\mathcal{R}}$ maps $\mathcal{R}$ into itself, we can estimate the distance $| \Psi_{q}(\beta_q(z))-\Psi_q(z)|$ with the diameter of the rectangle $\Psi_q(\mathcal{R})$. Hence by \autoref{lem:sizes} 
\begin{align}
    | \Psi_{q+1}(z)-\Psi_q(z)|
    &=
    | \Psi_{q}(\beta_q(z))-\Psi_q(z)|
    \lesssim
    \max\{x_{q-1},y_{q-1}\}
    \lesssim
    H^q.
\end{align}

Since $H<1$ and $z$ is arbitrary, the sequence $\{ \Psi_q \}_{q \in \NN}$ is Cauchy in $L^\infty_x$ and the limit $\Psi$ is a continuous function. A similar argument proves the same for $\{ \Psi_q^{-1} \}_{q \in \NN}$, namely
\begin{align}
   | \Psi_{q+1}^{-1}(x)-\Psi_q^{-1}(x)|
    &=
    | \beta_q^{-1}(\Psi_{q}^{-1}(x))-\Psi_q^{-1}(x)|
    \lesssim
    \max\{s_{q-1},h_{q-1}\}
    \lesssim
    W^q,
\end{align}
and therefore $\Psi_q^{-1}$ converges towards a continuous limit $\Phi$.

Notice that the limits $\Psi,\Phi$ take values in $\mathcal{A}$, $\TT \times I$ respectively since the maps $\Psi_q,\Psi_q^{-1}$ coincide with the identity outside $\mathcal{R}_0,\Psi_q(\mathcal{R}_0)$ respectively.

Let us check that $\Phi = \Psi^{-1}$.
For $x\in\mathcal A$ let $z_q:=\Psi_q^{-1}(x) \to \Phi(x)$ and $\Psi_q(z_q)=x$. Therefore, by continuity of $\Psi$ and the uniform convergence $\Psi_q \to \Psi$ we have
\begin{align}
    |\Psi(\Phi(x))-x| 
    &\leq 
    |\Psi(\Phi(x))-\Psi(z_q)|+|\Psi(z_q)-\Psi_q(z_q)| \to 0,
\end{align}
and thus $\Psi(\Phi(x))=x$.
A similar argument proves $\Phi(\Psi(z))=z$ and hence $\Phi = \Psi^{-1}$.

Finally, let us verify that $\Psi$ is measure-preserving (the proof for $\Psi^{-1}$ being analogous).
Since $\Psi_q \to \Psi$ uniformly and $\Psi_q$ is measure-preserving for every $q$, it holds for every continuous function $f : \mathcal{A} \to \R$ 
\begin{align}
\int_{\TT \times I}f(\Psi(z))\,dz   
=
\lim_{q \to \infty}
\int_{\TT \times I}f(\Psi_q(z))\,dz
=
\int_{\mathcal{A}} f(x)\,dx.
\end{align}
Equality against every continuous $f$ uniquely identifies the pushforward measure $\Psi_\sharp( \mathbf{1}_{\TT \times I}\mathscr{L}^2) = \mathbf{1}_\mathcal{A} \mathscr{L}^2$, and the proof is complete.
\end{proof}

\section{Construction of the Hamiltonian}
\label{sec:Hamiltonian}
In this section we define the Hamiltonian $H$ and the associated velocity field $u$ of \autoref{thm:main}. 

Recall the definition of the annulus-shaped regions $\mathcal{A} \subset \tilde{\mathcal{A}} \subset \R^2$ from \autoref{ssec:Psi_0}.
Introduce a smooth cutoff function $\chi_{\mathcal{A}}$ such that $\chi_\mathcal{A} \equiv 1$ on $\bar{\mathcal{A}}$ and $\chi_\mathcal{A} \equiv 0$ on $\tilde{\mathcal{A}}^c$.  
Moreover, let $\pi_2 : \TT \times I \to I$ be the projection on the second component $\pi_2(s,h)=h$.

Since by construction $\Psi_q \equiv \tilde{\Psi}_0$ in a neighborhood of $\partial \mathcal{A}$ for every $q \in \NN$ and $\tilde{\Psi}_0$ is a smooth diffeomorphism by \autoref{lem:Psi_0.diffeo}, we can define a smooth Hamiltonian on the whole plane by
\begin{align} \label{eq:definition.Hq}
    H_q(z) := 
    \begin{cases}
\pi_2(\Psi_q^{-1}(z)), &\mbox{ if } z \in \mathcal{A},
    \\
     \pi_2(\tilde{\Psi}_0^{-1}(z)) \chi_{\mathcal{A}}(z) ,
     &\mbox{ if } z \in \tilde{\mathcal{A}} \setminus \mathcal{A},
     \\
     0,
     &\mbox{ otherwise}.
    \end{cases}
\end{align}
and the associated velocity field $u_q:=\nabla^\perp H_q$. Notice that both $H_q$ and $u_q$ are compactly supported in $\tilde{\mathcal{A}} \subset \R^2$. 

Next, we check that the restriction $u_q|_{\mathcal{A}} : \mathcal{A} \to \R^2$ of the velocity field on $\mathcal{A}$ satisfies
\begin{align} 
u_q|_{\mathcal{A}} =\partial_s\Psi_q \circ \Psi_q^{-1}.
\end{align}

Differentiating the identity $H_q(\Psi_q(s,h))=h$ with respect to the two variables $s,h$ gives
\begin{align} \label{eq:nablaHq}
\nabla H_q(\Psi_q(s,h)) \cdot \partial_h\Psi_q (s,h) = 1,
\quad
\nabla H_q(\Psi_q(s,h)) \cdot \partial_s\Psi_q (s,h) = 0.
\end{align}
By \eqref{eq:det=1} and observing that the maps $\beta_q$ are measure-preserving, we have
\begin{align} \label{eq:nablaHq.bis}
\det (\partial_h\Psi_q (s,h),\partial_s\Psi_q (s,h)) 
= 
1.
\end{align}
By elementary linear algebra, the only vector $\nabla H_q(\Psi_q(s,h))$ satisfying \eqref{eq:nablaHq} and \eqref{eq:nablaHq.bis} is
\begin{align}
   \nabla H_q(\Psi_q(s,h)) = -J \partial_s\Psi_q (s,h) .
\end{align}
From the above we obtain
\begin{align} \label{eq:H=nabla.perp.u}
u_q(\Psi_q(s,h))
=
\nabla^\perp H_q(\Psi_q(s,h))
 =
 J(-J \partial_s\Psi_q (s,h))
 =
\partial_s\Psi_q (s,h).
\end{align}

This relation also identifies the associated flow for points in $\mathcal{A}$. 
Indeed, for every $(s,h) \in \TT \times I$ and $t \geq 0$, 
\begin{align}
  \partial_t   
  \Psi_q(s+t,h)
  =
  \partial_s\Psi_q(s+t,h)
  =
  u_q(\Psi_q(s+t,h)),
\end{align}
and therefore the associated flow is simply a translation of the first argument in the parametrization $\Psi_q$:
\begin{align} \label{eq:flow}
 \Phi_q(t,\Psi_q(s,h))=\Psi_q(s+t,h).
\end{align}

\begin{lemma}
    For every $q \in \NN$ it holds
\begin{align} \label{eq:difference.Hq+1-Hq}
    \| H_{q+1}-H_q \|_{L^\infty} \lesssim h_{q-1}, 
\end{align}
where $h_{q-1}$ is defined in \autoref{lem:sizes}.
\end{lemma}

\begin{proof}
First of all notice that $H_{q+1}|_{\mathcal{A}^c} \equiv H_{q}|_{\mathcal{A}^c}$ and therefore we only have to bound the difference on $\mathcal{A}$. 

Let $x \in\mathcal A$ and $(s,h)=z=\Psi_q^{-1}(x)$. Recalling $\Psi_{q+1} = \Psi_q \circ \beta_q$ we have $\Psi_{q+1}^{-1} = \beta_q^{-1} \circ \Psi_q^{-1}$. 
Since $H_q := \pi_2 \circ \Psi_q^{-1}$ and $H_{q+1} := \pi_2 \circ \Psi_{q+1}^{-1}$ it holds
\begin{align} \label{eq:difference.Hamiltonian}
    H_{q+1}(x)-H_q(x) = \pi_2(\beta_q^{-1}(z))-\pi_2(z).
\end{align}

Arguing similarly to the proof of \autoref{lem:homeo}, if $\beta_q^{-1}(z) \neq z$ then there exists a rectangle $\mathcal{R} \in \mathscr{P}_{q-1}$, such that $z, \beta_q^{-1}(z) \in \mathcal{R}$.
Hence by \eqref{eq:difference.Hamiltonian} and \autoref{lem:sizes} it holds
\begin{align} 
    \| H_{q+1}-H_q \|_{L^\infty} \lesssim |\pi_2(\beta_q^{-1}(z))-\pi_2(z)| \lesssim h_{q-1}. 
\end{align}
\end{proof}

The next proposition shows that the sequence of Hamiltonians $\{H_q\}_{q \in \NN}$ is Cauchy in $C^{1,\alpha}_c(\R^2)$ and, as such, converges towards a limit $H \in C^{1,\alpha}_c(\R^2)$.
This will be the Hamiltonian of the velocity field
\begin{align}
    u := \nabla^\perp H \in C^{\alpha}_c(\R^2,\R^2) 
\end{align}
that is the object of \autoref{thm:main}.

\begin{proposition}
    The sequence $\{H_q\}_{q \in \NN}$ is Cauchy in $C^{1,\alpha}(\R^2)$ and the limit $H \in C^{1,\alpha}(\R^2)$ is compactly supported on $\tilde{\mathcal{A}}$ and its restriction on $\mathcal{A}$ satisfies
\begin{align}
    H|_\mathcal{A} = \pi_2 \circ \Psi^{-1}
    .
\end{align}
\end{proposition}

\begin{proof}
Recall that by construction $H_q$ is smooth and independent of $q$ in a neighborhood of $\mathcal{A}^c$, thus we only have to prove the statement in $\mathcal{A}$.
Let $x \in \mathcal{A}$ and recall \eqref{eq:difference.Hamiltonian}:
\begin{align}
H_{q+1}(x)-H_q(x) &= \pi_2(\beta_q^{-1}(\Psi_q^{-1}(x)))-\pi_2(\Psi_q^{-1}(x)).  
\end{align}
As in \autoref{lem:homeo}, both $\Psi_q^{-1}(x)$ and $\beta_q^{-1}(\Psi_q^{-1}(x))$ must either coincide or belong to some rectangle $\mathcal{R} \in \mathscr{P}_{q-1}$ in the $(q-1)$-th affine generation, and $\Psi_{q}^{-1}$ acts affinely on $x$.
Tracking back the expansions and contractions due to the maps $\beta_j^{-1}$, for $j=0,\dots,q-1$ and using \eqref{eq:Psi.affine}, we can explicitly compute the differential of $\Psi_q^{-1}$:
\begin{align}
    D (\Psi_q^{-1})(x)
    =
    \begin{pmatrix}
        V^q  & 0
        \\
        0 & V^{-q} 
    \end{pmatrix}
    D(\Psi_0^{-1})(x) = 
    \begin{pmatrix}
        V^q  & 0
        \\
        0 & - V^{-q} 
    \end{pmatrix}
\end{align}

Therefore, by the chain rule, \eqref{eq:bounds.B.delta.R}, and \autoref{lem:sizes} we have the following estimate on the derivatives
\begin{align}
    |\partial_1 (H_{q+1}(x)-H_q(x))|
    &\lesssim
    V^q \| \partial_1(\beta_q^{-1})^2 \|_{L^\infty}
    \lesssim 
    V^q \frac{h_{q-1}}{s_{q-1}} (q+q_0)^{pM_1}
    \lesssim \left( \frac{H}{W}\right)^q (q+q_0)^{pM_1},
\\
    |\partial_2 (H_{q+1}(x)-H_q(x))|
    &\lesssim
    V^{-q} (1+\| \partial_2(\beta_q^{-1})^2 \|_{L^\infty})
    \lesssim 
    V^{-q} (q+q_0)^{pM_1},
\end{align}
and keeping only dominant terms we get
\begin{align}
    \| H_{q+1}- H_q \|_{C^1} 
    &\lesssim 
 \left(\frac{H}{W}\right)^q (q+q_0)^{p M_1}.
\end{align}

Next, we compute the second derivatives of $H_{q+1}- H_q$. 
First of all notice that $\Lambda_q$ is bounded from below, uniformly in $q \in \NN$, by our choice of parameters \eqref{eq:definition.delta.q}. Thus we have
\begin{align}
|\partial_1\partial_1 (H_{q+1}(x)-H_q(x))|
    &\lesssim
    V^{2q} \| \partial_1\partial_1(\beta_q^{-1})^2 \|_{L^\infty}
    \lesssim 
    V^{2q} \frac{h_{q-1}}{s_{q-1}^2} (q+q_0)^{pM_2}
    \lesssim \left( \frac{2HV}{W^2}\right)^q (q+q_0)^{pM_2} ,
    \\
|\partial_1\partial_2 (H_{q+1}(x)-H_q(x))|
    &\lesssim
     \| \partial_1\partial_2(\beta_q^{-1})^2 \|_{L^\infty}
    \lesssim 
     \frac{1}{s_{q-1}} (q+q_0)^{pM_2}
    \lesssim \left( \frac{2}{W}\right)^q (q+q_0)^{pM_2}  ,
    \\
|\partial_2\partial_2 (H_{q+1}(x)-H_q(x))|
    &\lesssim
    V^{-2q} \| \partial_2\partial_2(\beta_q^{-1})^2 \|_{L^\infty}
    \lesssim 
    V^{-2q} \frac{1}{h_{q-1}} (q+q_0)^{pM_2}
    \lesssim \left( \frac{2}{HV}\right)^q (q+q_0)^{pM_2},
\end{align}
and therefore, by our choice of parameters:
\begin{align}
    \| H_{q+1}- H_q \|_{C^2}
    &\lesssim 
\left( \frac{2HV}{W^2}\right)^q (q+q_0)^{pM_2} .     
\end{align}

By interpolation, we deduce for some finite constant $N$
\begin{align}
    \| H_{q+1} - H_q \|_{C^{1,\alpha}} 
    &\lesssim
    \| H_{q+1}- H_q \|_{C^1} ^{1-\alpha}
    \| H_{q+1}- H_q \|_{C^2} ^{\alpha}
    \lesssim
    \left(\frac{2H V^\alpha}{W^{1+\alpha}}\right)^q (q+q_0)^{N}.
\end{align}
By assumption on $W,H,V$ the ratio on the right-hand side is strictly smaller than $1$.
In particular, the sequence $\{H_q\}_{q \in \NN}$ is Cauchy in $C^{1,\alpha}(\mathcal{A})$ and therefore it converges to some $H \in C^{1,\alpha}(\mathcal{A})$. In fact, the convergence holds in $C^{1,\alpha}_c(\R^2)$ and $H$ vanishes outside $\tilde{\mathcal{A}}$ since each $H_q$ does.

Finally, since $\Psi_q^{-1}\to \Psi^{-1}$ uniformly, passing to the limit in \eqref{eq:definition.Hq} we obtain
\begin{align}
    H|_\mathcal{A} = \pi_2 \circ \Psi^{-1}
    .
\end{align}
\end{proof}

Let us conclude this section with an estimate on the flow deformation at each step $q \in \NN$, which will be needed in the next section to show absence of anomalous dissipation for the velocity field $u$.
\begin{lemma}
Let $\Phi_q$ be given by \eqref{eq:flow}, then there exists a finite constant $N$ such that 
   \begin{align} \label{eq:key.bound.deformation}
    \int_0^1 \int_{\mathcal{A}} |D \Phi_q(t,x)|^2 \, dx \, dt
    \lesssim (q+q_0)^{4N}
    V^{2q}.
\end{align} 
\end{lemma}
\begin{proof}

Let us define the matrix $M_q : \TT \times I \to \R^{2\times 2}$ as follows
\begin{align}
    M_q := D\Psi_q = 
    \begin{pmatrix}
        \partial_s \Psi_q^1 & \partial_h \Psi_q^1
        \\
        \partial_s \Psi_q^2 & \partial_h \Psi_q^2
    \end{pmatrix}.
\end{align}
Next, let us recall the definition of the affine bulks $\mathcal{P}_q$ from \eqref{eq:nested.bulks} in \autoref{sec:parametrization}. If a point $z \in \TT \times I \setminus \mathcal{P}_0$ then $\beta_q(z) \equiv z$ for every $q \geq 1$ and 
\begin{align} 
    M_q(z) = D\Psi_1(z), \quad
    \forall z \in \TT \times I \setminus \mathcal{P}_0,
    \quad
    \forall q \geq 1.
\end{align}

On the other hand, for every $z \in \cap_{q \geq 0} \mathcal{P}_q$, every map $\beta_q$ is affine on $z$ and therefore
\begin{align}
   |\partial_s\Psi_q(z)| &\lesssim  V^{-q}, \label{eq:dsPsi.upper.bound1}
   \\
   |\partial_h\Psi_q(z)| &\lesssim  V^{q}.
\end{align}
Finally, for $z \in \mathcal{P}_0 \setminus \bigcap_{q \geq 0} \mathcal{P}_{q}$ there exists
\begin{align}
j(z) := \min\{ j \geq 0 \, : \, z \in  \mathcal{P}_j \setminus \mathcal{P}_{j+1}  \}. 
\end{align}
The integer $j(z)$ represents the last generation of affine bulks that contains the point $z$. As such, the maps $\beta_0,\beta_1,\dots,\beta_{j(z)}$ act affinely on $z$, while the map $\beta_{j(z)+1}$ does not, and all the maps $\beta_q$ with $q>j(z)+1$ act on $z$ as the identity.
Consequently, using the chain rule and 
\begin{align}
    D\Psi_q(z) = D\Psi_{j(z)+2}(z)
    =
    D \Psi_{j(z)+1} (\beta_{j(z)+1}(z))  D \beta_{j(z)+1}(z),
    \quad
    \forall q>j(z)+2.
\end{align}
By definition of $j(z)$ the map $\Psi_{j(z)+1}$ acts affinely on $\beta_{j(z)+1}(z)$ with differential
\begin{align}
    D \Psi_{j(z)+1} (\beta_{j(z)+1}(z))
    =
    \begin{pmatrix}
        V^{-j(z)-1} & 0
        \\
        0 & -V^{j(z)+1}
    \end{pmatrix},
\end{align}
thus by \autoref{lem:sizes} and \eqref{eq:bounds.B.delta.R} we deduce that for some finite constant $N$ it holds for every $q \in \NN$
\begin{align}
    |\partial_s\Psi_q(z)| 
    &\lesssim \max_{j<q} \frac{V^{-j}s_j + V^j h_j} {s_j} (j+q_0)^{N}
 \lesssim (q+q_0)^{N}, \label{eq:dsPsi.upper.bound2}
 \\
 |\partial_h \Psi_q(z)| 
    &\lesssim \max_{j<q} \frac{V^{-j}s_j + V^j h_j} {h_j} (j+q_0)^{N}
 \lesssim 
 (q+q_0)^{N} V^q.
\end{align}

Moreover, since every $\Psi_q$ is area-preserving
\begin{align}
    1= |\det M_q| \leq |\partial_s\Psi_q| |\partial_h\Psi_q|
\end{align}
and from the lines above we deduce for every $z$:
\begin{align} \label{eq:dPsi.bounds}
  (q+q_0)^{-N}V^{-q} \lesssim |\partial_s\Psi_q(z)| &\lesssim (q+q_0)^{N},
  \\ \label{eq:dPsi.bounds2}
  |\partial_h\Psi_q(z)| &\lesssim (q+q_0)^{N}V^q .  
\end{align}

We are now ready to give estimates on the deformation of the flow $\Phi_q$. Recall that by \eqref{eq:flow} we have $\Phi_q(t,\Psi_q(s,h))=\Psi_q(s+t,h)$.

Fix $z = (s,h)$ and $t \geq 0$ and let us define $s_0 := s$, $s_1 = s+t$, and for $i \in \{0,1\}$
\begin{align}
    \sigma_i := |\partial_s \Psi_q(s_i,h)| \neq 0,
    \quad
    \tau_i := \frac{\partial_s \Psi_q(s_i,h)}{\sigma_i}. 
\end{align}
Choose a versor $n_i$ orthogonal to $\tau_i$ such that $(n_i,\tau_i)$ is positively
oriented, and rewrite the vector $\partial_h \Psi_q(s_i,h)$ in coordinates with respect to this basis
\begin{align}
    \partial_h \Psi_q(s_i,h)
    =
    [\partial_h \Psi_q(s_i,h)]^n n_i + [\partial_h \Psi_q(s_i,h)]^\tau\tau_i.
\end{align}
Since $\det (M_q) \equiv -1$ we have 
\begin{align}
    [\partial_h \Psi_q(s_i,h)]^n = \sigma_i^{-1}.
\end{align}
Writing the matrix of change of coordinates as
\begin{align}
    R_i := \begin{pmatrix}
        n_i^1 & \tau_i^1 
        \\
        n_i^2 & \tau_i^2 
    \end{pmatrix}
\end{align}
allows to write
\begin{align}
    M_q(s_i,h) = R_i \tilde{M}_q(s_i,h),
    \quad
    \tilde{M}_q(s_i,h)
    := \begin{pmatrix}
        0 & \sigma_i^{-1} 
        \\
       \sigma_i & [\partial_h \Psi_q(s_i,h)]^\tau
        \end{pmatrix}.
\end{align}

By \eqref{eq:flow} we get
\begin{align}
    D \Phi_q(t,\Psi_q(s,h))
    &= 
    M_q(s+t,h)  [M_q(s,h)]^{-1}
    \\
    &=
    R_1 \tilde{M}_q(s_1,h)[\tilde{M}_q(s_0,h) ]^{-1} R_0^{-1}, 
\end{align}
where the matricial product above can be explicitly computed as
\begin{align}
    \tilde{M}_q(s_1,h)[\tilde{M}_q(s_0,h) ]^{-1}
    =
    \begin{pmatrix}
      \sigma_0/\sigma_1 & 0
      \\
      \sigma_0[\partial_h \Psi_q(s_1,h)]^\tau-\sigma_1[\partial_h \Psi_q(s_0,h)]^\tau
      & \sigma_1/\sigma_0
    \end{pmatrix}.
\end{align}
From \eqref{eq:dPsi.bounds} and \eqref{eq:dPsi.bounds2}  we can bound
\begin{align}
   |D \Phi_q(t,\Psi_q(s,h))| \lesssim (q+q_0)^{2N} V^{q}. 
\end{align}
Since $\Psi_q$ is measure preserving for every $q \in \NN$, the above line yields \eqref{eq:key.bound.deformation}.
\end{proof}

\section{Absence of anomalous dissipation}
\label{sec:dissipation}

Let $K \in C^{1,\alpha}_c(\R^2)$ be a Hamiltonian and let us denote $S^\kappa_K : \R_+ \times L^2_x \to L^2_x$ the solution operator associated to \eqref{eq:advection-diffusion} with $u = \nabla^\perp K$, namely
\begin{align}
    S^\kappa_K(t,\theta_0) := \theta_t^\kappa.
\end{align}

The following lemma gives an estimate of the $L^2_x$ operator norm of the difference $S^\kappa_K(t,\cdot)-S^\kappa_G(t,\cdot)$, in terms of the $L^\infty$ distance between the two Hamiltonians $K,G \in C^{1,\alpha}_c(\R^2)$.

\begin{lemma}
\label{lem:hamiltonian-comparison}
Let $K,G \in C^{1,\alpha}_c(\R^2)$. Then for every $\theta_0 \in L^2_x$, $\kappa>0$, and $t\geq0$ it holds
\begin{align}
  \|S_K^\kappa(t,\theta_0 )-S_G^\kappa(t,\theta_0 )\|^2_{L^2_x}
  \leq
  \frac{\|K-G\|^2_{L^\infty}}{2\kappa^2} \| \theta_0 \|^2_{L^2_x}.
\end{align}
\end{lemma}

\begin{proof}
Let $\theta_0 \in L^2_x$ be fixed and consider 
\begin{align*}
    \theta_t^{\kappa,K} := S_K^\kappa(t,\theta_0),
    \quad
    \theta_t^{\kappa,G} := S_G^\kappa(t,\theta_0),
    \quad
    \bar{\theta}_t^\kappa := \theta_t^{\kappa,K}-\theta_t^{\kappa,G}.
\end{align*}
By definition, $\bar{\theta}^\kappa$ solves the equation
\begin{align} \label{eq:difference}
    \partial_t \bar{\theta}^\kappa
    +
    \nabla^\perp  K \cdot \nabla \bar{\theta}^\kappa
    +
    (\nabla^\perp  K-\nabla^\perp  G)\cdot\nabla \theta^{\kappa,G}
  =
  \kappa\Delta\bar{\theta}^\kappa.
\end{align}

Next, observe that since $\nabla^\perp  K-\nabla^\perp  G$ is divergence-free we can rewrite
\begin{align}
(\nabla^\perp  K-\nabla^\perp  G) \cdot \nabla \theta^{\kappa,G}
=
\nabla^\perp  (K-G) \cdot \nabla \theta^{\kappa,G}
=
\mathrm{div} ((K-G)J^T\nabla\theta^{\kappa,G}).
\end{align}

Plugging this expression into \eqref{eq:difference} we get
\begin{align}
\partial_t \bar{\theta}^\kappa
    +
    \nabla^\perp  K \cdot \nabla \bar{\theta}^\kappa
    +
    \mathrm{div} ((K-G)J^T\nabla\theta^{\kappa,G})
  =
  \kappa\Delta\bar{\theta}^\kappa.    
\end{align}

Energy balance, integration by parts, and Young's inequality give
\begin{align}
\|  \bar{\theta}^\kappa_t \|_{L^2_x}^2  + \kappa \int_0^t \| \nabla  \bar{\theta}^\kappa_s\|^2_{L^2_x} ds   
&\leq
\frac{\| K-G\,\|_{L^\infty_x}^2}{\kappa} \int_0^t \| \nabla  \theta^{\kappa,G}_s\|^2_{L^2_x} ds  
\\
&\leq
\frac{\| K-G\,\|_{L^\infty_x}^2}{2\kappa^2} \| \theta_0\|_{L^2_x}^2,
\end{align}
concluding the proof.
\end{proof}

If $K$ is a Hamiltonian of class $ K\in C^2_c(\R^2)$, then the inviscid transport equation \eqref{eq:advection-diffusion} with $\kappa=0$ admits a unique solution
\begin{align}
   S^0_K(t,\theta_0) := \theta_t^0
\end{align}
which is given by the composition of the initial condition $\theta_0$ with the inverse of the flow map
\begin{align}
\partial_t \Phi(t,x)=\nabla^\perp K( \Phi(t,x)),
  \qquad
  \Phi(0,x)=x.   
\end{align}
Moreover, the map $x \mapsto \Phi(s,x)$ is measure-preserving for every $s \geq 0$, and therefore for every $s \geq 0$ and $x \in \R^2$
\begin{align*}
\mathrm{det}(D\Phi)=1,
\quad
D \Phi := \begin{pmatrix} 
  \partial_1 \Phi^1 & \partial_2 \Phi^1
  \\
  \partial_1 \Phi^2 & \partial_2 \Phi^2
  \end{pmatrix},
\end{align*}
and the transpose matrix $(D \Phi)^T$ can be explicitly inverted with inverse
\begin{align} \label{eq:inverse.transpose}
    (D\Phi)^{-T} = \begin{pmatrix} 
  \partial_2 \Phi^2 & -\partial_1 \Phi^2
  \\
  -\partial_2 \Phi^1 & \partial_1 \Phi^1
  \end{pmatrix}.
\end{align}

The following lemma compares the viscous and inviscid solutions $S^\kappa_K(t,\theta_0)$ and $S^0_K(t,\theta_0)$ in terms of the Euclidean norm of the gradient of the flow map $\Phi$, when the initial condition $\theta_0$ is smooth.

\begin{lemma}
\label{lem:viscous-inviscid}
Let $ K\in C^2_c(\R^2)$ and $\theta_0\in C_c^\infty(\R^2)$. Then for every $\kappa>0$ and $t\geq0$ it holds
\begin{align} 
  \|S^\kappa_K(t,\theta_0)-S^0_K(t,\theta_0)\|_{L^2_x}^2
  \leq
  \kappa \|\nabla \theta_0\|_{L^\infty_x}^2
  \int_0^t\int_{\mathrm{supp}(\nabla \theta_0)}
  | D \Phi(s,x)|^2 dx \,ds.
\end{align}

\end{lemma}

\begin{proof}
Denote $\bar{\theta}^\kappa_t := S_K^\kappa(t,\theta_0) - S_K^0(t,\theta_0)$, which solves
\begin{align}
    \partial_t \bar{\theta}^\kappa + \nabla^\perp K \cdot \nabla \bar{\theta}^\kappa = \kappa \Delta \bar{\theta}^\kappa + \kappa \Delta \theta_t^0 . 
\end{align}
By energy balance and integrating by parts we get
\begin{align} \label{eq:bound.difference.viscous-inviscid}
\| \bar{\theta}^\kappa_t \| _{L^2_x}^2
+
\kappa \int_0^t \|\nabla\bar{\theta}^\kappa_s \|_{L^2_x}^2 ds
&\leq
\kappa \int_0^t \|\nabla \theta^0_s \|_{L^2_x}^2 ds.
\end{align}
Let us rewrite the right-hand side of the expression above as follows. First of all, notice that $\theta^0_s(\Phi(s,x))=\theta_0(x)$ for every $s \geq 0$ and $x \in \R^2$, and hence
\begin{align}
    \nabla \theta_0(x) = (D \Phi(s,x))^T \cdot\nabla \theta^0_s(\Phi(s,x)),
\end{align}
and since $\Phi$ is measure preserving and recalling the expression \eqref{eq:inverse.transpose} we can rewrite
\begin{align*}
    \| \nabla \theta^0_s \|_{L^2_x}^2
    =
    \| \nabla \theta^0_s \circ \Phi(s,\cdot) \|_{L^2_x}^2
    &=
    \| (D\Phi)^{-T}(s,\cdot) \nabla \theta_0 \|_{L^2_x}^2
    \\
    &\leq
    \| \nabla \theta_0 \|_{L^\infty_x}^2 \int_{\mathrm{supp}(\nabla \theta_0)} |(D\Phi)^{-T}(s,x)| ^2 dx
    \\
    &=
    \| \nabla \theta_0 \|_{L^\infty_x}^2 \int_{\mathrm{supp}(\nabla \theta_0)} |D\Phi(s,x)| ^2 dx.
\end{align*}
Plugging this bound into \eqref{eq:bound.difference.viscous-inviscid} we obtain the desired result.
\end{proof}

Putting together \autoref{lem:hamiltonian-comparison} and \autoref{lem:viscous-inviscid} we deduce the following criterion for absence of anomalous dissipation.

\begin{proposition}
\label{prop:criterion.no.diss}
Let $\{\kappa_q\}_{q \in \NN} \subset (0,1)$ be a sequence such that $\kappa_q \downarrow 0$, and let
\begin{align}
    H \in C^{1,\alpha}_c (\R^2),
    \quad
    \{ H_q \}_{q \in \NN} \subset  C^2_c(\R^2)
\end{align}
be Hamiltonians such that
\begin{align} \label{eq:criterion.bound.hamiltonian}
\frac{\|H-H_q\|_{L^\infty}}{\kappa_q} \to 0,    
\end{align}
and the flow $\Phi_q$ associated to $H_q$ satisfies, for every compact $K\subset\R^2$:
\begin{align} \label{eq:criterion.bound.deformation}
  \kappa_q \int_0^1 \int_K
  |D\Phi_q(s,x)|^2\,dx\,ds \to 0.
\end{align}
Then for every $\theta_0 \in L^2_x$ it holds
\begin{align}
\kappa_q\int_0^1
  \| \nabla S_H^{\kappa_q}(s,\theta_0)\|_{L^2_x}^2 ds\to 0,
  \quad
  \mbox{as }
  q \to \infty.
\end{align}
\end{proposition}

\begin{proof}
By density and using that $\{S^{\kappa_q}_H(s,\cdot)\}_{q \in \NN}$ are uniformly bounded operators in $L^2_x$, we can suppose without loss of generality that $\theta_0 \in C^\infty_c(\R^2)$.
In this case, \autoref{lem:hamiltonian-comparison} and \autoref{lem:viscous-inviscid} give for every $t \in [0,1]$
\begin{align*}
  \| S_H^{\kappa_q}(t,\theta_0)-S_{H_q}^0(t,\theta_0) \|_{L^2_x}^2
  &\lesssim
  \| S_H^{\kappa_q}(t,\theta_0)-S_{H_q}^{\kappa_q}(t,\theta_0) \|_{L^2_x}^2
+
\| S_{H_q}^{\kappa_q}(t,\theta_0)-S_{H_q}^0(t,\theta_0) \|_{L^2_x}^2
\\
&\lesssim
  \frac{\|H-H_q\|_{L^\infty}^2}{\kappa_q^2 } \|\theta_0\|_{L^2_x}^2
  +
  \kappa_q \|\nabla \theta_0\|_{L^\infty_x}^2
  \int_0^t\int_{\mathrm{supp}(\nabla \theta_0)}
  | D \Phi_q(s,x)|^2 dx \,ds,
\end{align*}
whose right-hand side tends to zero by assumption.
Since $H_q \in C^2(\R^2)$ for every $q \in \NN$ it holds $\|S_{H_q}^0(t,\theta_0) \|_{L^2_x}^2 \equiv \|\theta_0 \|_{L^2_x}^2$ and thus $\| S_H^{\kappa_q}(t,\theta_0) \|_{L^2_x}^2 \to \|\theta_0 \|_{L^2_x}^2$ as well, when $q \to \infty$, for every fixed time $t \in [0,1]$. By energy balance we deduce the desired result.
\end{proof}

\begin{corollary} \label{cor:diss}
    For the velocity field $u$ constructed in \autoref{sec:Hamiltonian}, it holds for every $\theta_0 \in L^2_x$ and every sequence $\{\kappa_n\}_{n \in \NN} \subset (0,1)$ with $\kappa_n \downarrow 0$: 
\begin{align}
\kappa_n\int_0^1
  \| \nabla \theta^{\kappa_n}_s\|_{L^2_x}^2 ds\to 0,
  \quad
  \mbox{as }
  n \to \infty.
\end{align}
\end{corollary}

\begin{proof}
Since $V/H>V^2$, choose a constant $c$ such that
\begin{align}
  \frac1{\log V-\log H}
  <c<
  \frac1{2\log V},
\end{align}
and for any $\kappa \in (0,1)$ fix an integer $q=q(\kappa)$ 
\begin{align}
  q(\kappa):=\left\lfloor c\log\frac1\kappa\right\rfloor.
\end{align}
Let now $\kappa_n\downarrow0$ be an arbitrary sequence and put $q:=q(\kappa_n)$.

We need to check conditions \eqref{eq:criterion.bound.hamiltonian} and \eqref{eq:criterion.bound.deformation} from \autoref{prop:criterion.no.diss}.
As for \eqref{eq:criterion.bound.hamiltonian} we have by \eqref{eq:difference.Hq+1-Hq}
\begin{align}
    \| H-H_q \|_{L^\infty}
    \lesssim
    \sum_{j \geq q} \| H_{j+1}-H_j \|_{L^\infty}
    \lesssim
    \sum_{j \geq q} \left( \frac{H}{V} \right)^j 
    \lesssim 
    \left( \frac{H}{V} \right)^q.
\end{align}
Therefore
\begin{align}
    \frac{\| H-H_q \|_{L^\infty}}{\kappa_n}
    \lesssim
    \left( \frac{H}{V} \right)^q e^{\log\frac{1}{\kappa_n}}
    =
    e^{\log\frac{1}{\kappa_n} -q(\log V - \log H)} 
    \lesssim
    e^{\log\frac{1}{\kappa_n} (1-c (\log V - \log H))} 
    \to 0.
\end{align}

As for \eqref{eq:criterion.bound.deformation}, for every compact $K \subset \R^2$ we have by \eqref{eq:key.bound.deformation}
\begin{align}
    \kappa_n 
    \int_0^1 \int_{K} |D \Phi_q(t,z)|^2 \, dz \, dt
    &\lesssim
    \kappa_n \left( 1+ \int_0^1 \int_{K \cap \mathcal{A}} |D \Phi_q(t,z)|^2 \, dz \, dt \right)
    \\
    &\lesssim
    \kappa_n (q+q_0)^{4N} V^{2q}
    \\
    &=
    (q+q_0)^{4N} e^{2q \log V - \log\frac{1}{\kappa_n}}
    \lesssim
    (q+q_0)^{4N} e^{\log\frac{1}{\kappa_n}(2 c \log V -1)} \to 0. 
\end{align}

\end{proof}

\section{Failure of the weak Sard property}
\label{sec:weaksard}

By construction \eqref{eq:definition.affinebulk}, the affine bulk
$\mathcal{P}_0$ is a cartesian product
\begin{align} 
\mathcal{P}_0 = U_x \times U_y,  
\end{align}
where $U_x$ and $U_y$ are finite unions of disjoint intervals with positive length.
By induction, the affine bulk $\mathcal{P}_q$ at each step $q \in \NN$ is given by the disjoint union of rectangles in the $q$-th affine generation $\mathscr{P}_q$:
\begin{align} 
\mathcal{P}_q =: U^q_x \times U^q_y.  
\end{align}
Define
\begin{align}
    U_x^\infty := \bigcap_{q \geq 0} U^q_x,
    \quad
    U_y^\infty := \bigcap_{q \geq 0} U^q_y,
    \quad
    \mbox{and}
    \quad
    \mathcal{P}_{\infty} := \bigcap_{q \geq 0}\mathcal{P}_q
    =
    U_x^\infty \times U_y^\infty.
\end{align}
At each step of the iteration, the set $U^q_x$ (resp. $U^q_y$) retains a fraction $\lambda_{\delta_q}$ of the measure of $U^{q-1}_x$ (resp. $U^{q-1}_y$). 
Consequently,
\begin{align}
    \mathscr{L}^1(U^q_x) = \frac{\Lambda_{q}}{\lambda_{\delta_0}} \mathscr{L}^1(U^0_x),
    \quad
    \mathscr{L}^1(U^q_y) = \frac{\Lambda_{q}}{\lambda_{\delta_0}} \mathscr{L}^1(U^0_y),
\end{align}
and continuity from above of the Lebesgue measure gives
\begin{align}
    \mathscr{L}^1(U^\infty_x) = \mathscr{L}^1(U^0_x) \lim_{q \to \infty}\frac{\Lambda_{q}}{\lambda_{\delta_0}}  > 0,
    \quad
    \mathscr{L}^1(U^\infty_y) = \mathscr{L}^1(U^0_y) \lim_{q \to \infty}\frac{\Lambda_{q}}{\lambda_{\delta_0}}  > 0.
\end{align}
In the above line we have used that $\lim_{q \to \infty} \Lambda_q >0$ by our choice of parameters.

Each point $z \in \mathcal{P}_{\infty}$ has the property that $\beta_q$ acts affinely on $z$ for every $q \in \NN$, and \eqref{eq:DBdelta} holds. Therefore $|\partial_s\Psi_q(z)| \lesssim V^{-q}$ and therefore $|u_q(\Psi_q(z))| \to 0$ as $q \to \infty$, for every $z \in \mathcal{P}_{\infty}$.
Since $u_q \to u$ and $\Psi_q \to \Psi$ uniformly, and $u$ is uniformly continuous, we have
\begin{align}
 |u(\Psi(z))|
 &\leq
 |u(\Psi(z))-u(\Psi_q(z))|
 +
 |u(\Psi_q(z))-u_q(\Psi_q(z))|
 +
 |u_q(\Psi_q(z))| \to 0,
\end{align}
from which we deduce $|u(\Psi(z))|=0$.
Therefore 
\begin{align}
    \Psi(\mathcal{P}_\infty) \subset S,
\end{align}
where $S$ denotes the set of critical points of $H$.

\begin{proposition}
\label{prop:sard}
If a point $(s,h) \in \mathcal{P}_\infty$, then $\Psi(s,h)$ belongs to a connected component of the level set $H^{-1}(\{h\})$ with positive Hausdorff measure.
Moreover, one has
\begin{align} \label{eq:pushforw}
  H_\sharp\left( \mathbf{1}_{\Psi(\mathcal{P}_\infty) } \mathscr{L}^2
  \right)
  =
  \mathscr{L}^1 \left(U^\infty_x\right) \mathbf{1}_{U^\infty_y}(h)\,d h.
 \end{align} 
\end{proposition}

\begin{proof}
If $z = (s,h) \in \mathcal{P}_\infty$, then $\Psi(z) \in H^{-1}(\{h\})$. Therefore the level sets of the Hamiltonian contain 
\begin{align}
    \Psi(\TT \times \{h\}) \subset H^{-1}(\{h\}) 
\end{align}
which is homeomorphic to $\TT$, and thus is connected and must have positive Hausdorff $\mathscr{H}^1$ measure.

In order to show \eqref{eq:pushforw}, let $A\subset\R$ be a Borel set.  
Since $H(\Psi(s,h))=h$, it holds
\begin{align}
    \Psi^{-1} \left( 
    \Psi(\mathcal{P}_\infty) \cap H^{-1}(A)
    \right)
    &=
    \mathcal{P}_\infty \cap \left( \TT \times (A \cap I) \right)
    \\
    &=
    \left(U^\infty_x \times U^\infty_y \right) \cap \left( \TT \times (A \cap I) \right)
    \\
    &=
    U^\infty_x \times \left( U^\infty_y \cap A \right) .
\end{align}
Since $\Psi$ is measure preserving by \autoref{lem:homeo}, we have
\begin{align}
H_\sharp\left( \mathbf{1}_{\Psi(\mathcal{P}_\infty) } \mathscr{L}^2
  \right)(A)
  &=
  \mathscr{L}^2 (\Psi(\mathcal{P}_\infty) \cap H^{-1}(A))
  \\
  &=
  \mathscr{L}^2 \left( 
  U^\infty_x \times \left( U^\infty_y \cap A \right)
  \right)
  \\
  &=
  \mathscr{L}^1 \left(U^\infty_x\right) 
  \mathscr{L}^1 \left( U^\infty_y \cap A \right). 
\end{align}
\end{proof}
Since $\mathscr{L}^1 \left(U^\infty_x\right)>0 $ and $\mathscr{L}^1 \left(U^\infty_y\right)>0$, the measure $H_\sharp\left( \mathbf{1}_{\Psi(\mathcal{P}_\infty) } \mathscr{L}^2  \right)$ is absolutely continuous with respect to $\mathscr{L}^1$ and non-trivial. In particular
\begin{corollary} \label{cor:fail.sard}
The Hamiltonian $H$ fails the weak Sard property.
\end{corollary}

To conclude, for completeness we give the:
\begin{proof}[Proof of \autoref{thm:main}]
By the construction of \autoref{sec:Hamiltonian} we have a velocity field
\begin{align}
    u := \nabla^\perp H \in C^{\alpha}_c(\R^2,\R^2) 
\end{align}
whose Hamiltonian fails the weak Sard property by previous \autoref{cor:fail.sard}. Nonetheless, the solution to the advection-diffusion equation \eqref{eq:advection-diffusion} does not have anomalous dissipation by \autoref{cor:diss}. 
\end{proof}

\bibliography{biblio}{}

@article{AlBiCr13,
  author  = {Alberti, Giovanni and Bianchini, Stefano and Crippa, Gianluca},
  title   = {Structure of Level Sets and {Sard}-Type Properties of {Lipschitz} Maps},
  journal = {Ann. Sc. Norm. Super. Pisa Cl. Sci. (5)},
  volume  = {12},
  number  = {4},
  pages   = {863--902},
  year    = {2013},
  doi     = {10.2422/2036-2145.201107_006}
}

@article{DrGoGu26,
  author  = {Dribas, Roman V. and Golovnev, Andrew S. and Gusev, Nikolay A.},
  title   = {On the Weak {Sard} Property},
  journal = {J. Math. Anal. Appl.},
  volume  = {555},
  number  = {1},
  pages   = {130022},
  year    = {2026},
  doi     = {10.1016/j.jmaa.2025.130022}
}

@article{DrGu26,
  author        = {Dribas, Roman V. and Gusev, Nikolay A.},
  title         = {On Chain Rule and Renormalization},
  year          = {2026},
  journal        ={arXiv 2606.08330},
}

@article{DrElIyJe22,
  author  = {Drivas, Theodore D. and Elgindi, Tarek M. and Iyer, Gautam and Jeong, In-Jee},
  title   = {Anomalous Dissipation in Passive Scalar Transport},
  journal = {Arch. Ration. Mech. Anal.},
  volume  = {243},
  number  = {3},
  pages   = {1151--1180},
  year    = {2022},
  doi     = {10.1007/s00205-021-01736-2}
}

@article{HeRo25,
  author        = {Hess-Childs, Elias and Rowan, Keefer},
  title         = {A Universal Total Anomalous Dissipator},
  year          = {2025},
  journal       = {arXiv 2501.18526},
}

@article{BuSzWu26,
  author        = {Burczak, Jan and Sz{\'e}kelyhidi, Jr., L{\'a}szl{\'o} and Wu, Bian},
  title         = {Scalar Anomalous Dissipation and Optimal Regularity via Iterated Homogenization},
  year          = {2026},
  journal       = {arXiv 2604.13912},
  primaryClass  = {math.AP},
  note          = {arXiv:2604.13912}
}

@article{Ro24,
  author  = {Rowan, Keefer},
  title   = {On Anomalous Diffusion in the {Kraichnan} Model and Correlated-in-Time Variants},
  journal = {Arch. Ration. Mech. Anal.},
  volume  = {248},
  number  = {5},
  pages   = {93},
  year    = {2024},
  doi     = {10.1007/s00205-024-02045-0}
}

@article {ElZl19,
    AUTHOR = {Elgindi, Tarek M. and Zlato\v s, Andrej},
     TITLE = {Universal mixers in all dimensions},
   JOURNAL = {Adv. Math.},
  FJOURNAL = {Advances in Mathematics},
    VOLUME = {356},
      YEAR = {2019},
     PAGES = {106807, 33},
      ISSN = {0001-8708,1090-2082},
   MRCLASS = {35Q30 (35F05 37B20)},
  MRNUMBER = {4008523},
MRREVIEWER = {Beno\^it\ P.\ Desjardins},
       DOI = {10.1016/j.aim.2019.106807},
       URL = {https://doi.org/10.1016/j.aim.2019.106807},
}

@article {Pa25,
    AUTHOR = {Pappalettera, Umberto},
     TITLE = {On measure-preserving selection of solutions of {ODE}s},
   JOURNAL = {Proc. Amer. Math. Soc.},
  FJOURNAL = {Proceedings of the American Mathematical Society},
    VOLUME = {153},
      YEAR = {2025},
    NUMBER = {5},
     PAGES = {2037--2051},
      ISSN = {0002-9939,1088-6826},
   MRCLASS = {34A05 (34A34 35A08 35L65)},
  MRNUMBER = {4881393},
MRREVIEWER = {Iskander\ A.\ Taimanov},
       DOI = {10.1090/proc/17097},
       URL = {https://doi.org/10.1090/proc/17097},
}

@article {BoCiCr24,
    AUTHOR = {Bonicatto, Paolo and Ciampa, Gennaro and Crippa, Gianluca},
     TITLE = {Weak and parabolic solutions of advection-diffusion equations
              with rough velocity field},
   JOURNAL = {J. Evol. Equ.},
  FJOURNAL = {Journal of Evolution Equations},
    VOLUME = {24},
      YEAR = {2024},
    NUMBER = {1},
     PAGES = {Paper No. 1, 16},
      ISSN = {1424-3199,1424-3202},
   MRCLASS = {35K15 (35B65 35D30 35Q35)},
  MRNUMBER = {4678868},
MRREVIEWER = {Lvqiao\ Liu},
       DOI = {10.1007/s00028-023-00919-6},
       URL = {https://doi.org/10.1007/s00028-023-00919-6},
}

@article {ElLi24,
    AUTHOR = {Elgindi, Tarek M. and Liss, Kyle},
     TITLE = {Norm growth, non-uniqueness, and anomalous dissipation in
              passive scalars},
   JOURNAL = {Arch. Ration. Mech. Anal.},
  FJOURNAL = {Archive for Rational Mechanics and Analysis},
    VOLUME = {248},
      YEAR = {2024},
    NUMBER = {6},
     PAGES = {Paper No. 120, 28},
      ISSN = {0003-9527,1432-0673},
   MRCLASS = {35K15 (35Q35 35Q49 76B03)},
  MRNUMBER = {4829551},
       DOI = {10.1007/s00205-024-02056-x},
       URL = {https://doi.org/10.1007/s00205-024-02056-x},
}

@article{AlBiCr14,
  AUTHOR  = {Alberti, G. and Bianchini, S. and Crippa, G.},
  TITLE   = {A uniqueness result for the continuity equation in two dimensions},
  JOURNAL = {J. Eur. Math. Soc.},
  VOLUME  = {16},
  YEAR    = {2014},
  NUMBER  = {2},
  PAGES   = {201--234},
  DOI     = {10.4171/JEMS/431},
  URL     = {https://doi.org/10.4171/JEMS/431}
}

@article{BaBoDeMa26,
  AUTHOR  = {Bagnara, M. and Boutros, D. W. and De Lellis, C. and Mayboroda, S.},
  TITLE   = {Regularity thresholds for anomalous dissipation and related phenomena in passive scalars},
  JOURNAL = {arXiv:2603.11466},
  YEAR    = {2026},
  URL     = {https://arxiv.org/abs/2603.11466}
}

@article {DrEy17,
    AUTHOR = {Drivas, Theodore D. and Eyink, Gregory L.},
     TITLE = {A {L}agrangian fluctuation-dissipation relation for scalar
              turbulence. {P}art {I}. {F}lows with no bounding walls},
   JOURNAL = {J. Fluid Mech.},
  FJOURNAL = {Journal of Fluid Mechanics},
    VOLUME = {829},
      YEAR = {2017},
     PAGES = {153--189},
       DOI = {10.1017/jfm.2017.567},
       URL = {https://doi.org/10.1017/jfm.2017.567},
}

@article{BuSzWu23,
      title={Anomalous dissipation and {E}uler flows}, 
      author={Burczak, Jan and Székelyhidi Jr., László and Wu, Bian},
      year={2023},
      journal={arXiv 2310.02934},
}

@article {HoPaZhZh25,
    AUTHOR = {Hofmanov\'a, Martina and Pappalettera, Umberto and Zhu,
              Rongcahn and Zhu, Xiangchan},
     TITLE = {Anomalous and total dissipation due to advection by solutions
              of randomly forced {N}avier-{S}tokes equations},
   JOURNAL = {Ann. Appl. Probab.},
  FJOURNAL = {The Annals of Applied Probability},
    VOLUME = {35},
      YEAR = {2025},
    NUMBER = {5},
     PAGES = {3119--3149},
      ISSN = {1050-5164,2168-8737},
   MRCLASS = {60H15 (35B40 35Q30 76F25)},
  MRNUMBER = {4975044},
MRREVIEWER = {Martin\ Ondrej\'at},
       DOI = {10.1214/25-aap2168},
       URL = {https://doi.org/10.1214/25-aap2168},
}

@article{CoCrSo23,
author = {Maria Colombo and Gianluca Crippa and Massimo Sorella},
title = {Anomalous Dissipation and Lack of Selection in the {O}bukhov–{C}orrsin Theory of Scalar Turbulence},
journal = {Ann. PDE},
volume = {9},
number = {21},
year  = {2023},
doi = {10.1007/s40818-023-00162-9},
}

@article{JoSo24+,
      title={Anomalous dissipation via spontaneous stochasticity with a two-dimensional autonomous velocity field}, 
      author={Johansson, Carl P. and Sorella, Massimo},
      year={2024},
      journal={arXiv:2409.03599},
}

@article {ArVi25,
    AUTHOR = {Armstrong, Scott and Vicol, Vlad},
     TITLE = {Anomalous diffusion by fractal homogenization},
   JOURNAL = {Ann. PDE},
  FJOURNAL = {Annals of PDE. Journal Dedicated to the Analysis of Problems
              from Physical Sciences},
    VOLUME = {11},
      YEAR = {2025},
    NUMBER = {1},
     PAGES = {Paper No. 2, 145},
       DOI = {10.1007/s40818-024-00189-6},
       URL = {https://doi.org/10.1007/s40818-024-00189-6},
}
\bibliographystyle{alpha}

\end{document}